\documentclass[sn-mathphys]{sn-jnl}

 \usepackage{proof}
 \usepackage{color}

\jyear{2021}%

\theoremstyle{thmstyleone}%
\newtheorem{theorem}{Theorem}
\newtheorem{proposition}[theorem]{Proposition}%

\theoremstyle{thmstyletwo}%

\theoremstyle{thmstylethree}%
\newtheorem{definition}{Definition}%

\newcommand{\impl}{\rightarrow}

\newcommand{\et}{\wedge}
\newcommand{\vel}{\vee}
\newcommand{\fal}{\bot}

\newcommand{\non}{\neg}

\newcommand{\GC}{\mathrm{G}}

\newcommand{\groi}{\blacktriangleright}

\newcommand{\grom}{\gg}

\newcommand{\grot}{\triangleright}

\newcommand{\ace}{\varepsilon}

\newcommand{\red}{\mapsto}

\newcommand{\uno}{$(i)$~}
\newcommand{\due}{$(ii)$~}
\newcommand{\tre}{$(iii)$~}
\newcommand{\quattro}{$(iv)$~}

\begin{document}

\title[Grounding Operators: Transitivity and Trees, Logicality and Balance]{Grounding Operators: Transitivity and Trees, Logicality and Balance}


\author*{\fnm{Francesco A.} \sur{Genco}}\email{frgenco@gmail.com}

%


%


\abstract{We formally investigate immediate and mediate grounding operators from an inferential perspective. We discuss the differences in behaviour displayed by several grounding operators and consider a general distinction between grounding and logical operators. Without fixing a particular notion of grounding or grounding relation, we present inferential rules that define, once a base grounding calculus has been  fixed, three grounding operators: an operator for immediate grounding, one for mediate grounding---corresponding  to the transitive closure of the immediate grounding one---and a grounding tree operator, which enables us to internalise chains of immediate grounding claims without loosing any information about them. We then present an in-depth proof-theoretical study of the introduced rules by focusing, in particular, on the question whether grounding operators can be considered as logical operators and whether balanced rules for grounding operators can be defined.}




\maketitle

\section{Introduction}
\label{sec:introdution}

The notion of grounding is usually conceived as an objective and explanatory relation that connects two relata---the {\it ground} and the {\it consequence}---if the first one determines or explains the second one.  In the contemporary philosophical literature, much effort has been devoted to analyse the formal aspects of grounding by logical systems, see for instance \cite{sch11, fin12, fin12b, cs12, cor14, pog16, pog18, korb18I, korb18II}, and these analyses often rely on characterisations of grounding by inferential calculi, see for instance \cite{sch11, fin12, cs12, cor14, pog16, pog18}. In most calculi, grounding is formalised by an operator acting on formulae. While much work has been devoted to the analysis  of specific notions of grounding and the study of specific grounding operators, no systematic study exists of the general formal features that grounding operators share. In this work, we endeavour in a first investigation of the formal behaviour of different grounding operators from an inferential perspective, by particularly focusing on the nature of grounding operators in general and on the relations entertained by immediate grounding and different formalisations of mediate grounding. Without fixing a particular notion of grounding or grounding relation, we study the proof-theoretical features of a generic immediate grounding operator and proof-theoretically investigate two different ways to generalise it to a mediate grounding operator.

In order to do so, we introduce three sets of inferential rules that, assuming that a grounding calculus has been fixed, define the behaviour of three grounding operators: an operator for immediate grounding, one for mediate grounding---corresponding  to the transitive closure of the immediate grounding one---and a grounding tree operator, which enables us to internalise chains of immediate grounding claims without loosing any information about them and their relations. Intuitively, immediate grounding connects a ground and a consequence when  the ground is directly linked in an explanatory way to the consequence. The immediateness of this kind of grounding connection can be differently spelled out depending one the particular grounding notion considered. It might, for instance, depend on the fact that the ground is one point simpler than the consequence with respect to the adopted complexity measure, or on the fact that the ground is more fundamental than the consequence of exactly one level according to a fixed hierarchy. According to most logical grounding notions, for instance, $A$ and $B$, taken together, are supposed to constitute an immediate ground of $A\et B$. Different notions of immediate grounding are discussed, for instance,  in \cite{sch11, fin12, cs12, pog16, pog18}. Mediate grounding, on the other hand, relates a ground and a consequence when they are linked by a chain of several immediate grounding steps. We have a simple example of mediate logical grounding, according to certain grounding notions, if we say that $A$, $C$ and $D$ constitute a ground of $(A\vel B)\et (C\et D)$. Indeed, if $A$ is an immediate ground of $A\vel B$, and $C$ and $D$ constitute an immediate ground of $C\et D$, then we can conclude that $A$, $C$ and $D$ constitute a mediate ground of $(A\vel B)\et (C\et D)$. Intuitively, we have the grounding connections displayed in the following tree-shaped diagram:
\[\deduce{(A\vel B)\et (C\et D)}{\deduce{\mid}{\deduce{A\vel B}{\deduce{\mid}{A}}}&\deduce{\mid }{\deduce{C\et D}{\deduce{\mid}{C} & \deduce{\mid}{D}}}}\]
While immediate grounding accounts for the direct links between a formula and the formulae immediately above it, the particular mediate grounding statement discussed above accounts for the explanatory relation between the root of the tree and its leaves. Different mediate grounding notions are discussed, for instance, in \cite{sch11, fin12, cor14, korb18I, pog20b}. Grounding trees, finally, are supposed to encode in a unique sentential object all information expressed by a tree-shaped diagram as the one showed above, see also \cite[§ 220]{bol14} for a similar diagrammatic representation of grounding trees. Instead of enabling us to express the mediate connection between a statement and any collection of statements explanatorily linked to it, as mediate grounding does, grounding trees enable us to express entire chains of explanatory steps leading from a collection of sentences to the  sentence that we wish to explain. 

Technically, immediate grounding will be formalised by the  $\groi$ operator, which  can only be introduced immediately after an immediate grounding rule has been applied. This grounding rule application will guarantee that the immediate grounding relation holds between the considered ground---corresponding to the premisses of the grounding rule---and the considered consequence---corresponding to the conclusion of the grounding rule. Mediate grounding will be formalised by the  $\grom$ operator, which internalises in the object language the transitive closure of the immediate grounding operator $\groi$. As we will show in order to characterise the mediate grounding operator, $\grom$ precisely enables us to select all, and only, the formulae that lie on a bar of a grounding derivation\footnote{As we will explain later, if we consider a grounding derivation as a progressive decomposition of its conclusion, then a bar of a  grounding derivation can be seen as a complete description of one stage of this decomposition.} and to use them in a mediate grounding statement for the conclusion of the grounding derivation. Finally, grounding trees will be formalised by nesting occurrences of the $\grot $ operator inside an occurrence of the $\groi $ operator. Thus we will be able to construct formulae  that exactly  correspond  to grounding derivations built by nesting several consecutive immediate grounding claims.

The rules that we will adopt for all grounding operators are fully modular and do not depend on a particular choice of background grounding calculus. As a consequence,  the presented work does not rely on the particular features of the considered grounding relation and applies to several of the  grounding relations that have been formally introduced in the literature. In particular, for any notion of grounding that can be formalised by grounding rules of the form $\vcenter{\infer{B}{A_1 & \ldots &A_n}}\,$---where $A_1 ,  \ldots , A_n$ are supposed to form a ground of $B$---we can define a grounding calculus and extend it by our rules in order to define the behaviour of a  grounding operator that exactly characterises the considered notion of grounding. Examples of grounding notions that can be formalised in this way are {\it full} and {\it partial} grounding as defined in \cite{sch11, fin12, cs12, cor14}, and {\it complete} grounding as defined in \cite{pog16, pog18}.

After having introduced and characterised our rules for the three grounding operators, we present an in-depth proof-theoretical investigation focusing on the question whether these rules can be considered as well-behaved definitions of the respective operators, and on establishing  what we can learn about the operators themselves by studying the presented inferential rules. In order to do so, we will, first, consider the question whether a grounding operator can be considered as a logical operator, and then---after having negatively answered this question---try to establish whether, nevertheless, the presented rules for grounding operators display a form of balance that enables us to conclude that they suitably define the operators. In order to do so, we will adopt 
methods coming from the  structuralist proof-theoretical approach to the characterisation of the notion of logical constant---see for instance \cite{dos80, dos89}---which dates back to the work of Koslow \cite{kos05} and Popper, see \cite{sh05}. We will consider in particular two traditional criteria of logicality, and show that while one is satisfied by the rules for grounding operators, the other one is not, unless a weaker version of it is considered and certain assumptions about the underlying immediate grounding relation hold.


The first criterion that we will consider corresponds to the criterion for sequent calculus rules called {\it deducibility of identicals} \cite{hac79} and  presented in \cite{pra71} for natural deduction rules under the name of {\it immediate expansion}, see \cite{np15} for a study of  this criterion and of a similar one discussed in \cite{bel62}. We will show, in particular, that a strict version of this criterion---the one originally employed in order to characterise logical operators---is not satisfied by any of the considered rules for grounding operators. After discussing the conceptual meaning of this failure and some connections of interest with certain essential features of the grounding relation, we argue that while this failure implies that our grounding operators do not comply with a standard construal of logical operators, it does not imply that these operators cannot undergo a meaningful analysis aimed at understanding whether their rules suitably define them as non-logical sentential operators. Such an analysis can indeed be conducted by finding a suitable way to loosen the criteria usually employed in the literature on inferential semantics.


The second considered criterion, which we call {\it detour eliminability}\footnote{We use this name in order to keep the terminology specific and not to employ words which are already overcharged of meanings}, requires that by deductively using a sentence constructed by applying the operator one does not obtain more information than that required to conclude that such a sentence is true. This precisely corresponds to the possibility of eliminating from any derivation any {\it detour} directly concerning the considered operator---that is, an inferential step introducing the operator immediately followed by one eliminating it. Detour eliminability is not only key to several criteria for the logicality of operators, it is also an essential requirement of normalisation results.  Moreover, it can be identified with the notion of proof-theoretic {\it harmony} presented in \cite{pra77} and is a central component of---or, at least, an effective means to test---most of the other notions of {\it harmony}, see for instance \cite{bel62, dum91, rea00, ten07} and \cite{pog10} for a survey of the literature on the issue.  If the rules for an operator enjoy both detour eliminability and deducibility of identicals, moreover, then they can be considered as an exact definition of the meaning of the operator in the sense that there is a perfect balance between the rules for deductively using the operator and the rules that determine when sentences constructed by applying the operator are true. As we will show, the rules for the immediate grounding  and grounding tree operators $\groi$ and $\grot$ admit the definition of Prawitz-style detour reductions which can be used to generalise the normalisation result in \cite{gen21}. On the other hand, while detour reductions can be defined also for the  mediate grounding  operator $\grom$---which implies that a local form of detour eliminability holds for $\grom$ as well---problems arise if we want to argue that a normalisation procedure extended by detour reductions for $\grom$ terminates. We will argue that these problems crucially obstacle the standard arguments that would be required to show that proper---or global---normalisation results for calculi containing the rules for $\grom $ hold. 

Finally, motivated by the fact that $\groi $ and $\grot $ are clearly well-behaved with respect to normalisation and, thus, detour eliminability, we propose and analyse a weaker version of the deducibility of identicals criterion which better suits---better than the traditional one---the hyperintensional nature of grounding.\footnote{An operator is hyperintensional if the truth of sentences built by using it is not necessarily preserved under the substitution of some of its arguments by logically equivalent arguments.} In dong this, we aim at showing that, even though grounding operators would not pass a logicality test, balanced sets of rules that suitably define their meaning can, in some cases, be found. The weaker version of deducibility of identicals constitutes, moreover, a criterion that does not trivialise the analysis of non-logical operators, and hence enables us to study the differences between different grounding operators. We will show, indeed, that while the immediate grounding and grounding tree operators $\groi$ and $\grot$ meet, under certain assumptions,  the weaker criterion,  the mediate grounding operator $\grom$ does not. We conclude by discussing the differences in behaviour between the first two operators and the third one from a conceptual perspective, stressing the parallel between the technical features of the operators and their intended interpretation.

\bigskip 

\noindent The rest of the article is structured as follows. In Section \ref{sec:language}, we introduce the language that we will employ and discuss the meaning of the introduced grounding operators. In Section \ref{sec:operators}, we present introduction and elimination rules for the grounding operators. In particular, in Section \ref{sec:transitive}, we present those for the mediate grounding operator and, in Section \ref{sec:trees}, those for the grounding tree operator. In Section \ref{sec:balance} we  investigate the proof-theoretical balance of the introduced rules for grounding operators: Section \ref{sec:intro-elim} will be devoted to detour eliminability and Section \ref{sec:elim-intro} to deducibility of identicals. In this section, we also  discuss the conceptual reasons and implications of the relation between the grounding operators and the considered proof-theoretical criteria. Finally, in Section \ref{sec:conclusions}, we present some concluding remarks and a discussion of possible ways to further develop the presented analysis.

\section{The language}\label{sec:language}

We begin by presenting the logical language that we will adopt, which includes the usual logical connectives, the operator $\groi$ for immediate grounding, the operator $\grom$ for mediate grounding, and the operator $\grot$ for constructing  grounding trees.
  
\begin{definition}[Formulae of the language]\label{def:lang}
{\small 
\begin{align*} \varphi \quad  ::=  \quad & \xi  \; \mid \;  \fal \; \mid \; \non
\varphi \; \mid \; \varphi \et \varphi \; \mid \; \varphi \vel \varphi
\; \mid \; \varphi \impl \varphi \;\mid 
\\ & 
 (\Psi \groi \varphi)  \;\mid \; (\Psi
[ \Psi ] \groi \varphi)\;\mid 
\\ & 
 (\Phi \grom \varphi)  \;\mid \; (\Phi [\Phi]  \grom \varphi)  \;\mid
\\ 
\psi \quad  ::= \quad &  \varphi \; \mid \; (\Psi )\grot \varphi    \;\mid \; (\Psi
[ \Psi ])\grot \varphi
\\ 
\xi \quad  ::= \quad &  p \; \mid \; q \; \mid
\; r \; \mid \ldots \end{align*}}where $\Phi$ is a list of the form  $\varphi , \ldots , \varphi$, $\Psi$ is a list of the form  $\psi , \ldots , \psi$, and  $p, q, r, \ldots $ are all propositional variables of the
language.
\end{definition} 

The grammar in Definition \ref{def:lang} enables us to construct any formula of the language of classical logic by using the non-terminal symbols $\varphi$ and $\xi$. 
Immediate grounding statements, on the other hand, can be constructed by employing $ (\Psi \groi \varphi ) $ and $(  \Psi
 [\Psi] \groi \varphi )$, and then by expanding $\Psi$ and $\varphi $ with formulae in the language of classical logic. For instance, if we indicate by  $\rightsquigarrow$ the expansion of one or more non-terminal symbols, we can have the following sequence of instantiations: $\varphi \rightsquigarrow (  \psi, \psi  [\psi] \groi \varphi )\rightsquigarrow (  \varphi , \varphi  [\varphi ] \groi \varphi ) \rightsquigarrow (  p, q 
 [r]  \groi \varphi \vel  \varphi) \rightsquigarrow  (  p, q 
 [r]  \groi \varphi \vel  (\varphi \et  \varphi)) \rightsquigarrow  (  p, q 
 [r]  \groi r \vel  (p \et q))$. Nevertheless, by  using the non-terminal symbol $\psi$, we can also nest an occurrence of $\grot$ to the left of an occurrence of $\groi$ and thus construct a grounding tree. Once we have constructed a subformula of the form   $ ( \Psi )\grot \varphi  $ or $ ( \Psi
 [\Psi]  )\grot \varphi$, we can either keep extending  the grounding tree by using the list $\Psi$ of non-terminal symbols $\psi$ to nest another operator $\grot$ to the left of the previous one, or we can stop nesting  $\grot$ and expand a non-terminal symbol $\psi $ in the list $\Psi$ as $\varphi$. An example of the latter possibility is the following  sequence of instantiations: $\varphi \rightsquigarrow (  \psi \groi \varphi ) \rightsquigarrow  (  (\psi)\grot \varphi \groi \varphi ) \rightsquigarrow   (  (\varphi )\grot p \groi q ) \rightsquigarrow     (  (r)\grot p \groi q )$. Notice that the subformulae  with $\grot$ as outermost operator are parenthesised in a different way with respect to those with  $ \groi $ or $\grom$ as outermost operator. We explain in the examples below why this is so. The construction of mediate grounding statements by $(\Phi \grom \varphi)  $ and $ (\Phi [\Phi]  \grom \varphi) $ is identical to that of immediate ones, except that the lists $\Phi $ of non-terminal symbols $\varphi$ do not enable us to nest $\grot$ to the left of $\grom$.

In the following, we will employ capital Latin letters as metavariables for formulae,  capital Greek letters as metavariables for multisets of formulae, and we will omit parentheses when no ambiguity arises. Moreover, with a slight abuse of notation, when we will write formulae of the form $\Gamma  [\Delta] \groi A $,  $\Gamma  [\Delta] \grom A $ and  $(\Gamma  [\Delta] )\grot A $, we will admit the possibility that they denote formulae of the form $\Gamma \groi A $,  $\Gamma \grom A $ and  $(\Gamma )\grot A $, respectively. We do so by simply admitting the possibility that $\Delta $ is the empty list.

Let us now discuss the intended meaning of the grounding formulae that can be constructed by employing our grammar. A formula of the form $\Gamma  [\Delta] \groi A $ where  neither $\Gamma $ nor $ \Delta $ contain any formula of the form $(\Theta  [\Sigma] )\grot B $ corresponds to an immediate grounding claim and expresses that $\Gamma $ constitutes an {\it immediate ground} of the {\it consequence} $A$ under the {\it condition} $\Delta$.
A formula of the form $\Gamma [ \Delta] \grom A $ corresponds to a mediate grounding claim and expresses that  $\Gamma $ constitutes a  {\it mediate ground} of the {\it consequence} $A$ under the {\it condition}  $\Delta$. For instance, supposing that\[p,q\groi p\et q \qquad p\et q \groi \non \non ( p\et q)\]are legitimate grounding claims in our system, then we have that\[p,q \grom \non \non ( p\et q)\]is a legitimate grounding claim too.
A formula of the form $\Gamma  [\Delta] \groi A $ where  $\Gamma $ or $ \Delta $ or both contain formulae of the form $(\Theta [\Sigma] )\grot B $ corresponds to a {\it grounding tree}. Such a grounding tree constitutes a complex account of the truth of $A$ by an orderly display of the grounding instances that we can construct from $A$ to reach simpler and simpler grounds. Intuitively, a grounding tree can be seen as a mediate grounding claim in which we also include  the information concerning the way in which the grounded statement is related to its mediate grounds. Or, in other words, in a grounding tree we keep track of all immediate grounding steps that justify a mediate grounding claim. For instance, supposing that\[r,s\groi r \et s \qquad r\et s\groi \non \non (r \et s) \qquad  \non \non (r\et s) [ \non t] \groi \non \non (r\et s) \vel t    \]are  legitimate immediate grounding claims in our system, then we have that\[ ((r,s)\grot r\et s )\grot \non \non (r\et s)  [\non t] \groi  \non \non(r\et s) \vel t  \]is a legitimate grounding claim too. This grounding tree intuitively corresponds to the following tree in which each edge represents the connection due to an immediate grounding relation instance:
\[\deduce{\non\non(r\et s)\vel t}{\deduce{\mid}{\deduce{\non\non(r\et s)}{\deduce{\mid}{\deduce{r\et s}{\deduce{\mid }{r} & \deduce{\mid }{s} }}}}&\deduce{\mid }{[\non t]}}\]
Notice that the parentheses around $\Gamma$ in an expression of the form $(\Gamma ) \grot A$ are meant to stress that only $A$---and not the whole expression $(\Gamma ) \grot A$---is a part of the immediate ground in which the expression $(\Gamma ) \grot A$ occurs. For instance, by writing $(r,s)\grot r\et s$ inside the formula $ ((r,s)\grot r\et s )\grot \non \non (r\et s)  [\non t] \groi  \non \non(r\et s) \vel t$ above, we stress that only $r\et s$ is the immediate ground of $\non \non (r\et s)$. Similarly, by writing $ ((r,s)\grot r\et s )\grot \non \non (r\et s)   $ we stress that, among the formulae occurring in this expression,  only $ \non \non (r\et s)  $ is a part of the immediate  ground of $  \non \non(r\et s) \vel t$.


For another example of a grounding tree, suppose that\[r [\non s]\groi r\vel s \qquad r\vel s\groi \non \non (r \vel s) \] are  legitimate immediate grounding claims in our system, then we have that
\[ (r[ \non s])\grot  r\vel s \groi \non\non (r\vel s)  \]is a legitimate grounding claim too. This grounding tree corresponds to the following tree of immediate grounding instances:
\[\deduce{\non\non(r\vel s)}{\deduce{\mid}{\deduce{r\vel s}{ \deduce{\mid}{r} & \deduce{\mid}{[\non s]}  }}}\]


\section{Rules for the grounding operators}
\label{sec:operators}

%
%

The introduction rule for the operator $\groi$ is presented in  Table \ref{tab:imm-rules}.\begin{table}[h]  \hrule 
\medskip

If \begin{itemize} \item $ \quad \vcenter{\infer={B}{A_1  & \ldots & A_n & [ C_1   & \ldots & C_m] }} \quad $ is a grounding rule application  
such that $A_1  , \ldots , A_n$ form the ground of $B$  under the possibly empty list of conditions  $ C_1, \ldots , C_m$

\item $\delta _1 , \ldots , \delta _n , \delta' _1 , \ldots , \delta' _m  $ are 
derivations of $A_1  , \ldots , A_n, C_1, 
\ldots , C_m$, respectively
\end{itemize}
then 

\begin{center}
$   \quad \vcenter{ \infer{A_1 , \ldots , A_n [  C_1   ,
\ldots , C_m] \groi B}{\infer={B}{\deduce{A_1}{\delta _1} &
\ldots & \deduce{A_n}{\delta _n} & [  \deduce{C_1}{\delta' _1}   &
\ldots & \deduce{C_m}{\delta' _m}]}}} \quad $ 
is a derivation
\end{center}\medskip

\hrule 
\caption{Introduction Rules for the Immediate Grounding Operator $\groi$}\label{tab:imm-rules}
\end{table}This rule 
reflects the idea that a  sentence with $\groi$ as outermost operator and no nested occurrences of $\grot$---that is, an immediate  grounding claim---can only be introduced on the basis of a legitimate grounding rule application---in order to have an immediate visual distinction, we use a double inference line when representing grounding rule applications. Technically, we can introduce $\groi$ only immediately below a grounding rule application. For instance, if we consider the grounding calculus in \cite{gpr21}, the following are legitimate grounding rule applications:\footnote{We slightly adapt the notation here and use square brackets instead of the bar between the premisses of the disjunction rule.}\[\infer={p\et q}{p&q}\qquad\qquad \infer={p\vel q}{p & [\non q]}\]Hence, by using our rules for the grounding operator, we can introduce $\groi$ as follows:\[\infer{p,q\groi p\et q}{\infer={p\et q}{p&q}}\qquad\qquad \infer {p[\non q]\groi p\vel q}{\infer={p\vel q}{p&[\non q]}}\]Thus we derive the grounding claim $p,q\groi p\et q$ under the hypotheses that $p$ and $q$ are true, and we derive the grounding claim $p[\non q]\groi p\vel q$ under the hypotheses that $p$ and $\non q$ are true.\footnote{Notice that, more in general, the conclusion of the $\groi$ introduction rule could be of the form $h(A_1) ,
\ldots , h(A_n) [ h(C_1),  \ldots ,  h(C_m)]\groi B$ where $h$ is a function from formulae to formulae that depends on the particular system in which the considered grounding rule is defined and on the derivations of the premisses of its application. This is due to the fact that in certain proof systems the premisses of a grounding rule application cannot always be directly interpreted as the grounds of its conclusion. Since this detail is irrelevant for the present work and can be easily handled given a specific proof system, we omit the function $h$ in the rule definition.}

The elimination rules for $\groi$ are presented in  Table \ref{tab:gro-el-rules}. The first three rules in this table correspond to what is called the {\it factivity}  of grounding; that is, the feature of grounding according to which all elements of a ground and the corresponding consequence are supposed to be true in case the grounding claim connecting them is true. The last rule in Table \ref{tab:gro-el-rules} enables us to derive the negation of those grounding claims that cannot be derived by the grounding rules in the chosen grounding calculus. This rule enables us to reason about the falsity of certain grounding claims if we exclude the possibility of having true grounding claims that cannot be derived in the calculus. Certain grounding calculi, nevertheless, are supposed only to provide a minimal framework for grounding which can be extended by further assumptions about true grounding claims or further grounding rules---see, for instance, \cite{fin12}. If this is the case,  it is possible to omit the last rule in Table \ref{tab:gro-el-rules} in order to obtain a more flexible calculus which admits extensions of this kind.

\begin{table}[h] \centering \hrule 
\medskip

\[\infer{B}{\Gamma [ \Delta]   \groi 
B}\qquad
\infer{A}{\Gamma _1 ,  A , \Gamma _2 [ \Delta] \groi 
B}
\qquad\infer{C}{\Gamma [ \Delta _1 , C , \Delta _2]    \groi B}\]where the outermost operator of $A$ and $C$ is not $\grot$\bigskip\bigskip

If there is no grounding rule application $ \quad \vcenter{\infer={B}{A_1 &
\ldots & A_n & [ C_1 &
\ldots & C_m] }} \quad $ then $
\quad \vcenter{ \infer{\fal}{A_1 , \ldots , A_n [
C_1   ,
\ldots , C_m] \groi B}}  \quad $ is a rule application

\medskip

\hrule 
\caption{Elimination Rules for  the Immediate Grounding Operator $\groi$}\label{tab:gro-el-rules}
\end{table}

\subsection{Mediate grounding operator}
\label{sec:transitive}

In Table \ref{tab:tra-intro}, we present the rules to introduce the mediate grounding operator $\grom$  on the basis of the immediate grounding operator $\groi$.

\begin{table}[h] \centering \hrule 
\smallskip

\[\infer{\Gamma  [  \Delta  ]\grom   A }{\Gamma  [  \Delta]  \groi   A} \qquad \qquad \infer{\Gamma _1  , \Gamma  , \Gamma _2 [  \Delta  , \Delta _1]   \grom   B }{\Gamma  [  \Delta]  \grom  A &&& \Gamma _1  , A , \Gamma _2 [  \Delta _1]  \grom B}
\]\[\infer{\Gamma _1  [  \Delta _1  ,  \Gamma , \Delta , \Delta _2]   \grom   B }{\Gamma  [  \Delta]  \grom  A &&& \Gamma _1 [  \Delta _1 , A , \Delta _2]  \grom B}
\]\hrule \smallskip

\caption{Introduction Rules for the Mediate Grounding Operator $\grom$}\label{tab:tra-intro}
\end{table}

The introduction rules for $\grom$ implement the obvious inductive definition of the transitive closure of $\groi$. In particular, the first rule of the table, corresponds to the base case, according to which any immediate grounding claim  $\Gamma [\Delta]\groi A $ is also  a mediate grounding claim $\Gamma [\Delta]\grom A$. The second rule in the table, on the other hand, enables us to compose two mediate grounding claims by transitivity. In particular, if $\Gamma [\Delta]$ constitutes a mediate ground of $A$ and $A$ is contained in a mediate ground of $B$, then we can insert $\Gamma $ instead of $A$ in the ground of $B$ and add $\Delta$ to the conditions on the ground of $B$.

In Table \ref{tab:tra-elim}, we present the elimination rules for $\grom$.
\begin{table}[h] \centering \hrule 
\smallskip
\[\vcenter{\infer{B}{\Gamma [ \Delta]   \grom 
B}}\qquad
\vcenter{ \infer{A}{\Gamma _1 , A , \Gamma _2 [ \Delta] \grom
B}}\qquad
\vcenter{ \infer{C}{\Gamma [ \Delta _1 , C , \Delta _2 ] \grom
B}}  \]\hrule \smallskip 
\caption{Elimination Rules for the Mediate Grounding Operator $\grom$}\label{tab:tra-elim}
\end{table}These rules to eliminate $\grom$ simply implement the factivity of mediate grounding. Indeed, given a mediate grounding claim, they enable us to derive the consequence, any part of the ground, and any part of the condition.

\subsubsection{Completeness of the mediate grounding rules}
\label{sec:completeness-tra}

We show now that mediate grounding rules always enable us to internalise mediate grounding claims that are supposed to hold with respect to the chosen grounding calculus. We show in particular that, if a grounding derivation of a formula $A$ can be constructed in our grounding calculus, then we can derive a claim that expresses that certain collections of formulae that have been used to construct the grounding derivation of $A$ are mediate grounds of $A$. We only talk of {\it completeness} for the mediate grounding rules, and not of {\it characterisation}, since the converse result---i.e., that a certain mediate grounding claim is derivable only if a suitable grounding derivation exists---cannot be proved here. Indeed, the latter result essentially depends on the particular features of the considered grounding calculus.\footnote{In order to prove that a mediate grounding claim is derivable only if a suitable grounding derivation exists, one could, for instance, show that a normalisation result holds for the considered calculus and that, as a consequence, the calculus enjoys the canonical proof property for $\grom$---that is, if a mediate grounding claim is provable, the last rule applied in its normal proofs is a $\grom $ introduction rule. It might be the case, though, that not all normal proofs of mediate grounding claims are canonical in the specific  calculus under consideration.} In order to discriminate the collections of formulae that constitute a suitable mediate ground of a formula $A$, we employ the notion of {\it bar of a grounding derivation}.  Intuitively, if we consider a grounding derivation as a progressive decomposition of its conclusion, then a bar of a  grounding derivation can be seen as a complete description of one stage of this decomposition. Consider, for instance, the following derivation and suppose that it constitutes a legitimate grounding derivation according to a given notion of grounding:
\[\infer={(p\et(q\vel r) ) \vel (s\et t )}{\infer={p\et (q\vel r )}{p&\infer={q\vel r }{q&r} } & \infer={s\et t }{t&s}}\]
One of the bars of  this derivation contains exactly $p,q,r,t $ and $s$, that is, all leaves of the derivation. This bar  represents the final stage of the decomposition: the stage at which we have actually decomposed all complex subformulae obtained by progressively decomposing the original formula $(p\et(q\vel r)) \vel (s\et t )$. But also the formulae $p, q\vel r$ and $s\et t $ constitute a bar of our derivation. This second bar represents the stage of the decomposition at which we have already decomposed $p\et(q\vel r) \vel ( s\et t )$ into $ p\et(q\vel r) $ and $s\et t $, and then we have decomposed  $ p\et(q\vel r) $  into  $ p$  and  $ q\vel r$, but we have neither decomposed $q\vel r$ nor $s\et t $ yet. In more general terms, a bar represents a stage of a decomposition of this kind in the sense that it contains all formula occurrences that we have obtained so far during the decomposition, but none of the formula occurrences that we have decomposed to obtain them. 

We fix the obvious notion of grounding derivation and then formally define what are its bars.

\begin{definition}[Grounding derivation]
A grounding derivation is a derivation which is constructed by exclusively applying grounding rules to a set of consistent hypotheses and which contains at least one rule application.
\end{definition}

\begin{definition}[Bar of a derivation]
Given any derivation $\delta$ composed of inferential steps such that each step has one or more formulae as premisses and one formula as conclusion, the derivation-tree $t(\delta)$ of $\delta$ is a tree such that there is a one-to-one correspondence between the nodes of $t(\delta)$ and the formula occurrences in $\delta$ that verifies the following conditions: 
\begin{itemize}
\item the root of $t(\delta)$ corresponds to the conclusion of $\delta$, and 
\item if a node $n$ of $t(\delta)$ corresponds to a formula occurrence $o$ of $\delta$ and $o$ is the conclusion of an inferential step $i$, then each children of $n$ corresponds to a distinct premiss of $i$ and each premiss of $i$ corresponds to a distinct children of $n$.\end{itemize}

A bar of a derivation $\delta$ is a set of formula occurrences in $\delta$ that does not contain the conclusion of $\delta$ and such that the corresponding nodes form a set that shares exactly one element with each path (that is, set of consecutive nodes) connecting the root of $t(\delta)$ with one of its leaves. \end{definition}

We are now ready to prove that any bar of any grounding  derivation of a formula $A$ corresponds to a derivable mediate grounding claim for $A$.
\begin{proposition}If $ \Gamma $ contains all grounds and  $\Delta$ all conditions occurring in a bar of a grounding derivation of $A$ in a fixed calculus $\kappa$, then the grounding claim $\Gamma [\Delta ]\grom A$ is derivable in any calculus that contains the rules of $\kappa$, the rules for $\groi$ and the rules for $\grom$.
\end{proposition}
\begin{proof}
The proof is by induction on the number of rule applications occurring in the grounding derivation of $A$. In the base case, only one grounding rule is applied to derive $A$---by definition of grounding derivation, if no rule is applied in a derivation, then it is not a grounding derivation. Hence, $\Gamma [\Delta ]\groi A$ is directly derivable and $\Gamma [\Delta ]\grom A$ is derivable from it. Moreover, $\Gamma \cup \Delta $ is the only bar of the derivation. Suppose now that if $ \Gamma '$ contains all grounds and $\Delta '$ all conditions occurring on the nodes of a bar of a grounding derivation of  $A'$ containing less than $n$ grounding rule applications, then $\Gamma ' [\Delta ' ]\grom A '$ is derivable. We show that this holds also for grounding derivations containing $n$ grounding rule applications. Suppose that there is a grounding derivation $\delta $ of $A$ containing $n$ rule applications and that $ \Gamma $ contains all grounds and $\Delta $ all conditions occurring in a bar of that derivation. Consider then the bottommost rule application in $\delta$ and let us call it $r$. We can then list the elements  $ B _1 ,\ldots , B _n  $ of the bar $ \Gamma \cup \Delta$ in such a way that  $ B _1 ,\ldots , B _m  $---for $m\leq n$---are premisses of $r$, and $ B _{m+1} ,\ldots , B _n  $ belong to bars of distinct grounding derivations of  those premisses of $r$ which are not listed in $ B _1 ,\ldots , B _m  $. Let us call $ C  _1 ,\ldots , C  _p  $ the premisses of $r$ which do not belong to the list $ B _1 ,\ldots , B _m  $. Hence---by neglecting for a moment the difference between grounds and conditions---we can picture our grounding derivation as follows:
\[\infer=[r]{A}{ \infer*{B _1}{} & \ldots  & \infer*{B _m}{} &  \infer*{C  _1}{\Theta _1} & \ldots &  \infer*{C  _p}{\Theta _p}  }\]where each  $\Theta _i $ is the bar of the derivation of $C_i$ that contains some of the elements of $B _{m+1} ,\ldots , B _n$ and $\Theta _1 \cup  \ldots \cup \Theta _p$ is the multiset $\{B _{m+1} ,\ldots , B _n\}$. By inductive hypothesis, we have that each grounding claim $\Theta  _i \grom C _i$---where some elements of $\Theta _i $ might be between square brackets---is derivable. Moreover, also the grounding claim $B _1 , \ldots  , B _m, C  _1, \ldots,   C  _p\grom A$---where some formulae among $B _1 , \ldots  , B _m, C  _1, \ldots,   C  _p$ might be between square brackets---is directly derivable from the conclusion of the grounding rule application $r$ by a $\groi $ introduction immediately followed by a $\grom $ introduction. Hence, we can derive the grounding claim $\Gamma [\Delta ]\grom A= B _1 , \ldots  , B _m, \Theta _1, \ldots, \Theta _{p}\grom A$ as follows:
\[\infer{B _1 , \ldots  , B _m, \Theta _1, \ldots, \Theta _{p}\grom A}{\infer*{\Theta  _p \grom C _p}{} & \infer*{B _1 , \ldots  , B _m, \Theta _1, \ldots, \Theta _{p-1}, C  _p\grom A}{\infer{}{
\infer*{\Theta  _2 \grom C _2}{} & \infer{B _1 , \ldots  , B _m, \Theta _1, \ldots,   C  _p\grom A }{\infer*{\Theta  _1 \grom C _1}{} & \infer{B _1 , \ldots  , B _m, C  _1, \ldots,   C  _p\grom A}{\infer{B _1 , \ldots  , B _m, C  _1, \ldots,   C  _p\groi A}{\infer=[r]{A}{\infer*{B _1}{} & \ldots  & \infer*{B _m}{} &  \infer*{C  _1}{\Theta _1} & \ldots &  \infer*{C  _p}{\Theta _p}}}} }}}}\]\end{proof}

\subsection{Grounding trees}
\label{sec:trees}

We present now the rules for constructing grounding trees. As discussed above, a grounding tree represents a concatenation of several immediate grounding steps, and we use occurrences of the  operator $\grot$ nested inside an occurrence of the operator $\groi$ in order to encode a grounding tree as a formula. For instance, if the grounding claims $\Gamma , A \groi B$ and $\Delta  \groi A$ hold, then the symbol $\grot $ enables us to compose the two grounding claims in one as follows: $ \Gamma ,( \Delta )\grot A   \groi
B$. This formula means that $B$ is grounded by $\Gamma , A$ and that the component $A$ of the ground of $B$ is in turn grounded by $\Delta$. 

We adopt a different notation---i.e., $\grot$ and the relative parenthesising---for the occurrences of $\groi$ nested inside grounding claims in order to avoid any ambiguity in the interpretation of formulae. Indeed, if we used $\groi $ both for the outermost grounding claim of  a grounding tree and for the nested ones,  in the formula $ \Gamma ,( \Delta \groi A )  \groi
B$ we could either have that the formula $\Delta \groi A$ itself is part of the ground of $B$,\footnote{For instance,  this could happen according certain notions of grounding if $\Gamma = \{p\}$ and $B = p\et (\Delta \groi A)$.} or we could have that $A$ is part of the ground of $B$ and $\Delta \groi A$ is a nested grounding claim by which we specify that $\Delta$ is an immediate ground of $A$. In order to solve this ambiguity, we distinguish nested grounding claims by using   $\grot$  for them instead of $\groi$. Hence, while the formula $ \Gamma ,( \Delta \groi A )  \groi
B$ means that the ground of $B$ contains the formula $\Delta \groi A$, the formula $ \Gamma ,( \Delta )\grot A  \groi
B$ means that the ground of $B$ contains the formula $A$, which in turn has ground $\Delta$. Intuitively, if we represent grounding claims as trees in which the children of a node stand for the grounds of the formula occurring as that node, $ \Gamma ,( \Delta \groi A )  \groi
B$ corresponds to the tree below on the left, while $ \Gamma ,( \Delta )\grot A   \groi
B$ corresponds to the tree below on the right:\begin{center}
\[\infer {B}{\deduce{\mid}{\Gamma} & \deduce{\mid}{\Delta\groi A} } \qquad \qquad \infer {B}{\deduce{\mid}{\Gamma} & \deduce{\mid}{\deduce{A}{\deduce {\mid }{\Delta}}} } \]
\end{center}No formula, therefore,  has the form $(\Delta )\grot A$: we only use $\grot$ to distinguish the subformulae of $\Sigma \groi B$ that are used as parts of a ground and the subformulae of $\Sigma \groi B$ that are immediate grounds---or conditions of immediate grounds---of $B$. 
The symbol $\grot$ can be indefinitely nested to construct complex grounding trees.

For instance, we can write\[((Z )\grot A  , B )\grot C \;,\;\;  (D  [E])
\grot I\;\; [ (E [F] )\grot G ]\;\;  \groi \;\; H \]to mean that $H$ is grounded 
by $C$ and $I$ under the condition $G$; that $G$ is in turn grounded  by $E$
under the condition $F$; that $I$ is grounded by $D$ under the condition $E$; and that  $C$ is grounded  by $A$ and $B$; and, finally, that the ground $A$ of $C$ is in turn grounded  by
$Z$. What is expressed by such a formula could be represented by the following tree:
\[\infer{H}{\deduce{\mid}{\infer{C}{\deduce{\mid}{\deduce{A}{\deduce{\mid}{Z}}} & \deduce{\mid}{B}}} & \deduce{\mid}{\infer{I}{\deduce{\mid}{D} & \deduce{\mid}{[E]}}} & \deduce{\mid}{\infer{[G]}{\deduce{\mid}{[E} & \deduce{\mid}{[F]]}}} }\]
where, as in the trees above, the children of a node stand for the elements of the ground of the formula occurring as that node, and conditions are between square brackets.\footnote{We will not use this as a formal notation but just as a visual device to have a clearer grasp of the structure of grounding trees constructed by $\groi$ and $\grot$. Notice that, in the subtree rooted at $[G]$, both the ground $E$ and the condition $[F]$ are enclosed together between square brackets. We adopt this convention for the children of a node corresponding to a condition---such as $[G]$---in order to have a visual indication that the grounds of a condition should not be counted among the mediate grounds of the consequence.}

Nesting  $\grot$ inside $\groi$, thus,  does not correspond to using a grounding claim as part of the ground of some formula, but  serves the purpose of constructing chains of grounding claims. As already mentioned, a chain of grounding claims is similar to a mediate grounding claim but with the essential difference that, while a mediate grounding claim does not contain any information about the immediate grounding claims that justify it, a grounding tree contains all the information about the immediate grounding claims on which it is based.

\begin{table}[h] \centering \hrule 
\smallskip

\[ \infer{\Gamma _1 ,( \Delta [ \Theta]  )\grot A  , \Gamma _2 [  \Xi] \groi
B}{\Delta [  \Theta] \groi A &&& \Gamma _1 , A , \Gamma _2 [  \Xi] \groi B}
\qquad \qquad 
\infer{\Gamma  [  \Xi _1 ,  (\Delta [ \Theta]  )\grot C  , \Xi_2]  \groi
B}{\Delta [ \Theta]  \groi C &&& \Gamma [  \Xi _1 , C , \Xi_2]  \groi B}   
\]
\hrule \smallskip
\caption{Introduction Rules for the Grounding Tree Operator $\grot$}\label{tab:tree-intro}
\end{table}

The introduction rules for $\grot$ are presented in Table \ref{tab:tree-intro}. Intuitively, these rules enable us to plug grounding sentences inside other grounding sentences in a coherent way. For instance, if a formula $D$ can be
grounded by the complex ground $A, C$ and thus we can derive $A, C
\groi D$ and, moreover, $C$ can be grounded  by $B$ and thus we can derive $B\groi C$; then we can plug the second grounding claim into the
first one and obtain: $A, (B )\grot C \groi D$. Notice that, as explained above and stressed by the parentheses enclosing $B$, according to this notation, only $A$ and $C$ are part of the immediate ground of $D$; $B$ occurs in the formula as an immediate ground of $C$ and thus only as a {\it mediate} ground of $C$. Coherently,  the formula  $A, (B )\grot C\groi D$  can be read as follows: ``$D$ is immediately grounded by $A$ and $C$, and in turn $C$ is immediately  grounded by $B$''.

\begin{table}[h] \centering \hrule \smallskip
\[\infer{\Delta [ \Theta]  \groi  A}{\Gamma _1 , (\Delta [  \Theta]  )\grot  A , \Gamma _2 [  \Xi] \groi 
B}
\qquad\qquad \infer{\Delta [ \Theta]  \groi  C}{\Gamma [ \Xi_1 , (\Delta \mid \Theta  )\grot  C, \Xi_2] \groi  B}\]

\[\infer{\Gamma _1 , A , \Gamma _2 [  \Xi] \groi 
B}{\Gamma _1 , (\Delta [  \Theta]  )\grot  A , \Gamma _2 [  \Xi] \groi 
B}
\qquad\qquad \infer{\Gamma [ \Xi_1 , C, \Xi_2] \groi  B}{\Gamma [ \Xi_1 , (\Delta \mid \Theta  )\grot  C, \Xi_2] \groi  B}\]


\hrule \smallskip
\caption{Elimination Rules for the Grounding Tree Operator $\grot$}\label{tab:tree-el-rules}
\end{table}

The elimination rules for $\grot $ are presented in Table \ref{tab:tree-el-rules}. These rules enable us to simplify a grounding tree by eliminating one of its sub-grounding-trees. By applying them several times, for instance, we can extract all immediate grounding claims on which the grounding tree is based.


By an inductive definition, we make precise the idea that an occurrence of $\groi$ in a formula $F$ might form a grounding tree along with some occurrences of $\grot$ in $F$. This will be useful later.  In particular, we formally define the transitive relation that holds between an occurrence of $\groi$ or $\grot$ and all occurrences of $\grot$ that correspond to the nodes of the same grounding tree.
\begin{definition} We say that an occurrence $\star$ of $\groi$ or of $\grot$ {\it holds} an occurrence $\star '$ of $\grot$ if, and only if, 
\begin{itemize}
\item either  $\star $ is the outermost operator of a formula or subformula of the form \[ A_1, \ldots , A_n [ C_1, \ldots C_n]  \star B\] and   $\star '$ is the outermost operator of one of the subformulae $A_1, \ldots , A_n $, $C_1, \ldots C_n$;
\item or $\star $ holds an occurrence of $\grot$ that holds $\star '$.  
\end{itemize}  
\end{definition}

%
%

\subsubsection{Characterisation of the grounding tree rules}
\label{sec:characterisation-tree}

We show now that the rules for the grounding tree operator enable us to internalise in the object language any legitimate grounding derivation as a grounding tree. We will moreover show that, if a grounding tree is derivable from a consistent set of hypotheses, then we can construct a legitimate grounding derivation with exactly the same structure as the grounding tree.


In order to do so, we formally specify the intuitive correspondence between grounding trees and grounding derivations in an arbitrary grounding calculus.  
\begin{definition}[Grounding tree correspondence]
A grounding tree \[G_1 , \ldots , G_m [C_{m+1} \ldots C_n ] \groi A ,\] or subformula $( G_1 , \ldots , G_m [C_{m+1} \ldots C_n ] )\grot A $ of a grounding tree, corresponds to a grounding derivation $\delta$ if, and only if, 
\begin{itemize}\item the root of $\delta $ is $A$, 
\item the last rule application $r$ of $\delta $ has $n$ premisses (among which, say, the first $m$ are not between square brackets and the rest are), and 
\item for each $G_i$, where $1\leq i \leq n$, one of the following  holds:\begin{itemize}
\item  $G_i$ {\it does not} have as outermost operator an occurrence of $\grot$ held by the outermost occurrence of $ \groi $---respectively $\grot$---in $\Gamma [\Delta ] \groi A $---respectively $(\Gamma [\Delta ] )\grot A$---and the $i$th  premiss of $r$ is the formula $G_i$;
\item  $G_i$ {\it has} as outermost operator an occurrence of $\grot$ held by the outermost occurrence of $ \groi $ in $G_1 , \ldots , G_m [C_{m+1} \ldots C_n ] \groi A $---or of $\grot $ in  $(G_1 , \ldots , G_m [C_{m+1} \ldots C_n ] )\grot  A$---and the grounding derivation of the $i$th premiss of $r$ corresponds to $G_i$.
\end{itemize}
\end{itemize}
\end{definition}It is easy to see that, if we suppose that the premisses of a rule do not commute, each grounding derivation corresponds to exactly one grounding tree and vice versa. If we wish to allow for the commutation of rule premisses, we can still keep the one-to-one correspondence by considering commutative grounds and conditions.

We can prove now that the existence of a grounding derivation implies that the corresponding grounding tree can be derived by grounding rules and the rules for the grounding operators $\groi$ and $\grot$.   
\begin{proposition}\label{prop:compl-tree}If the grounding tree $\Gamma [\Delta ]\groi A$ corresponds to a legitimate grounding derivation in a fixed grounding calculus $\kappa$, then  $\Gamma [\Delta ]\groi A$  is derivable from the hypotheses $\Gamma ,  \Delta $ in  any calculus that contains the rules of $\kappa$, the rules for the immediate grounding operator, and  the rules for the  grounding tree operator.  
\end{proposition}
\begin{proof}
The proof is by induction on the number of rule applications occurring in the grounding derivation of $A$. In the base case, only one grounding rule is applied to derive $A$---indeed, if no rule is applied in the derivation, it is not a grounding derivation, but a logical one. Hence, $\Gamma [\Delta ]\groi A$ is derivable by our rules for the  immediate grounding operator. Suppose now that if a grounding tree $\Gamma ' [\Delta ' ]\groi A'$ corresponds to a legitimate grounding derivation in $\kappa$ that contains less than $n$ rule applications, then  $\Gamma ' [\Delta ' ]\groi A'$ is derivable in  $\kappa$ extended by our rules for immediate grounding and grounding trees. We show that this holds also for grounding derivations containing $n$ grounding rule applications. Suppose that the grounding derivation $\delta $ of $A$ containing $n$ rule applications corresponds to $\Gamma  [\Delta  ]\groi A$. Then we can consider the last rule applied in $\delta$ and we have that $\delta$  has one of the two following forms:\[\infer=[s]{A}{\infer*{\Pi _1}{} & \infer=[r]{B}{\infer*{\Sigma}{} &[\infer*{\Theta}{} ]}&\infer*{\Pi _2}{}&[\infer*{\Xi}{}]}\qquad \qquad \infer=[s]{A}{\infer*{\Pi}{} & [\infer*{\Xi_1}{} &  \infer=[r]{B}{\infer*{\Sigma}{} &[\infer*{\Theta}{}]}& \infer*{\Xi _2}{}]}\]If $\delta$ is of the form displayed above on the left, then, by inductive hypothesis, there is a grounding tree  $\Sigma ^{\star} [ \Theta ^{\star} ]\groi B$ derivable from the hypotheses  $\Sigma ^{\star} , \Theta ^{\star}$  which corresponds to the grounding derivation of $B$ and one grounding tree  $\Pi _1^{\star}, B, \Pi _2^{\star} [\Xi^{\star} ]\groi A$ derivable from the hypotheses $\Pi _1^{\star}, B, \Pi _2^{\star} , \Xi^{\star} $ which corresponds to our grounding derivation of $A$ in which we assume $B$ as a hypothesis rather than deriving it by $r$. Hence, by\[\infer{\Pi _1^{\star} , (\Sigma^{\star}  [ \Theta^{\star}  ])\grot B , \Pi _2^{\star}    [  \Xi^{\star}  ]  \groi A}{\Sigma ^{\star} [ \Theta ^{\star} ]\groi B &&& \Pi _1^{\star}, B, \Pi _2^{\star} [\Xi^{\star} ]\groi A}\] we can derive $\Pi _1^{\star} , (\Sigma^{\star}  [ \Theta^{\star}  ])\grot B , \Pi _2^{\star}    [  \Xi^{\star}  ]  \groi A$ from the hypotheses $\Sigma ^{\star} , \Theta ^{\star}, \Pi _1^{\star}, B, \Pi _2^{\star} , \Xi^{\star} $. But since $B$ can already be derived from the hypotheses $\Sigma ^{\star} , \Theta ^{\star}$, and, moreover,  $\Gamma [\Delta ] \groi A $ is supposed to correspond to $\delta$, we have that each element of $\Gamma [\Delta ]$ is suitably associated either to the relative premiss of $s$ or to its grounding derivation, which in turn, by induction hypothesis, corresponds to the relative element of $ \Pi _1^{\star}, (\Sigma ^{\star} [ \Theta ^{\star} ]\groi B), \Pi _2^{\star} [\Xi^{\star} ]$. Hence, by the fact that the correspondence between grounding trees and grounding derivations is one-to-one, we can conclude that $  \Gamma [\Delta ]  \groi A = \Pi _1^{\star} , (\Sigma^{\star}  [ \Theta^{\star}  ])\grot B , \Pi _2^{\star}    [  \Xi^{\star}  ]  \groi A $ and thus that we have a derivation of $\Gamma [\Delta ]  \groi A$ from the hypotheses $\Gamma , \Delta$.

If $\delta$ is of the form displayed above on the right, a similar argument will anyway lead us to the conclusion that $\Gamma [\Delta ]  \groi A$ is derivable from the hypotheses $\Gamma , \Delta$.\end{proof}

Finally, we show that, once we have fixed a grounding calculus, we can reconstruct the grounding derivation corresponding to any grounding tree which is derivable from a consistent set of hypotheses.\footnote{Notice that we do not prove here the obvious statement relying on the assumption that the grounding tree is provable, we prove a stronger statement relying only on the derivability of the gorunding tree from a consistent set of hypotheses. This is required if we want to give a good picture of the behaviour of grounding trees since, due to the factivity of grounding, grounding claims are usually supposed to depend on hypotheses concerning the truth of their constituents, and thus not provable but only derivable from consistent sets of hypotheses.}

\begin{proposition}For any consistent calculus  $\kappa^+$ defined by extending a grounding calculus $\kappa$ with the rules for immediate grounding and grounding tree operators, if  the grounding tree $\Gamma [\Delta ]\groi A $ is derivable in $\kappa^+$ from a consistent set of hypotheses, then there is a legitimate grounding derivation in $\kappa$ from a consistent set of hypotheses which exactly corresponds to $\Gamma [\Delta ]\groi A $.
\end{proposition}
\begin{proof}
Let us assume that $\delta $ is a $ \kappa^+$ derivation of $\Gamma [\Delta ]\groi A $ from a consistent set of hypotheses $\Pi$. Since $\Pi $ is consistent, $\kappa^+$ is consistent, and $\Gamma [\Delta ]\groi A $ is derivable from $\Pi$ in $\kappa^+$, $\fal $ cannot be derived from $\Gamma [\Delta ]\groi A $ in $\kappa^+$. In particular, if we consider the elements of the set $\{\Sigma _i [\Theta _i ]\groi B_i\}_{1\leq i \leq n}$ of immediate grounding claims which can be derived from  $\Gamma [\Delta ]\groi A $ by $\grot$ elimination rules, we have that $(i)$ $\Sigma _i [\Theta _i ]\groi B_i$ corresponds to a legitimate grounding rule application\[\infer={B_i}{\Sigma _i [\Theta_i ]}\]and $(ii)$ the set of formulae $\bigcup_0 ^n (\Sigma _i \cup \Theta _i )$ is consistent. Otherwise we could derive $ \Gamma [\Delta ]\groi A $ from $\Pi $ and then $\fal$ from $\Gamma [\Delta ]\groi A $ by $\grot$ elimination rules and $\groi$ elimination rules, which contradicts the assumption about the consistency of $\Pi$.  

But if $(i)$ and $(ii)$ hold, we can construct a grounding derivation in $\kappa$ from hypotheses $\bigcup_0 ^n (\Sigma _i \cup \Theta _i )$---or possibly from a subset of these hypotheses---by exclusively using the grounding rule applications\[\infer={B_i}{\Sigma _i [\Theta_i ]}\] The resulting derivation will exactly correspond to $ \Gamma [\Delta ]\groi A $. We prove this by induction on the number of occurrences of $\grot $ in $\Gamma [\Delta ]\groi A $. If  $\Gamma [\Delta ]\groi A $ does not contain any occurrence of $\grot$, it corresponds to a grounding rule application of the form\[\infer={A}{\Gamma [\Delta ]}\]which is exactly a grounding derivation in $\kappa$ which corresponds to $\Gamma [\Delta ]\groi A $ and only uses hypotheses among $\Gamma \cup \Delta$. We suppose then that for any grounding tree $\Gamma ' [\Delta ' ] \groi A'$ which contains less than $m>0$ occurrences of $\grot$, we can construct a grounding derivation $\delta '$  in $\kappa$ from a consistent set of hypotheses which corresponds to  $\Gamma ' [\Delta ' ] \groi A'$. We prove that this holds also for any  grounding tree $\Gamma  [\Delta  ] \groi A$ which contains $m$ occurrences of $\grot$. Since   $\Gamma  [\Delta  ] \groi A$ contains $m>0$ occurrences of $\grot$, there must be, for $1\leq k\leq p$, some elements  $(\Gamma _j ''  [\Delta _j  ''  ] )\grot A _j ''$ of $\Gamma , \Delta$ which clearly contain less than $m$ occurrences of $\grot$.   If we consider all grounding trees $\Gamma _j ''  [\Delta _j  ''  ] \groi A _j ''$---that is, the grounding trees which are identical to  $(\Gamma _j ''  [\Delta _j  ''  ] )\grot A _j ''$ except for the outermost operator---we have that, by inductive hypothesis,  there are grounding derivations $\delta _j ''$ in $\kappa$ which correspond to $\Gamma _j ''  [\Delta _j  ''  ] \groi A _j ''$ and which only depend on hypotheses which can be derived from immediate grounding claims which can, in turn, be derived from each $(\Gamma _j ''  [\Delta _j  ''  ] )\grot A _j ''$  by $\grot $ eliminations. Now, from $\Gamma  [\Delta  ] \groi A$ we can derive an immediate grounding claim of the form $\Gamma ^{\star}  [\Delta ^{\star}  ] \groi A$ where $\Gamma ^{\star}  , \Delta ^{\star}$ contain all elements of $\Gamma, \Delta$ which do not have as outermost operator $\grot $ and all formulae $A _j ''$ occurring in  the elements $(\Gamma _j ''  [\Delta _j  ''  ] )\grot A _j ''$  of $\Gamma, \Delta$ which do have as outermost operator $\grot $. This immediate grounding claim corresponds to a grounding rule of the form\[\infer={A}{\Gamma ^{\star} [\Delta ^{\star}]}\]and by composing this rule application to the conclusions $A _j ''$ of our grounding derivations $\delta _j ''$, we have our grounding derivation $\gamma $ in $\kappa$ which corresponds to  $\Gamma  [\Delta  ] \groi A$  and only depends on hypotheses which can be derived from immediate grounding claims which can, in turn, be derived from  $\Gamma  [\Delta  ] \groi A$ by $\grot $ eliminations. Indeed, as for the  hypotheses of $\gamma$, they  clearly have the required property since what can be derived from each  $(\Gamma _j ''  [\Delta _j  ''  ] )\grot A _j ''$ by $\groi$ and $\grot$ eliminations can also be derived from $\Gamma  [\Delta  ] \groi A$ by $\groi$ and $\grot$ eliminations. As for the correspondence between $\gamma$  and $\Gamma  [\Delta  ] \groi A$, we have that the root of $\gamma$ is exactly $A$; the last rule applied in $\gamma$ has the correct number of premisses without square brackets and within square brackets---by the definition of $ \Gamma ^{\star}$ and $ [\Delta ^{\star}]$; all premisses of the last rule applied in $\gamma$ without $\grot$ as outermost operator are identical to the elements of   $\Gamma  , \Delta $ without $\grot$ as outermost operator---again, by the definition of $ \Gamma ^{\star}$ and $ [\Delta ^{\star}]$;  and, finally, all premisses of the last rule applied in $\gamma$ with $\grot$ as outermost operator are derived by grounding derivations which correspond to the elements of   $\Gamma  , \Delta $ with $\grot$ as outermost operator---by the definition of $\gamma $, of the derivations $\delta _j''$, and of $ \Gamma ^{\star}$ and $ [\Delta ^{\star}]$.
\end{proof}

We conclude this section by stressing that in fully characterising the grounding tree operator, we also fully characterised the immediate grounding operator; indeed, in the base case, a grounding tree is an immediate grounding claim.

\section{Logicality and balance}
\label{sec:balance}

 
We will investigate now the proof-theoretical behaviour of the introduction and elimination rules that we presented for our three grounding operators and try to establish whether these rules  induce a well-behaved definition---in inferential terms---of the operators, and what this definition can tell us about the operators themselves. We will begin with a most general demarcation problem that arises in the study of inferential definitions of sentential operators: the logicality issue. In other words, we will address the question whether our grounding operators can be considered as logical operators. In order to do so, we will adopt methods coming from the structuralist proof-theoretical approach to the characterisation of the notion of logical constant---see for instance \cite{dos80, dos89}---which dates back to the work of Koslow \cite{kos05} and Popper, see \cite{sh05}. We will consider, in particular, two traditional criteria of logicality and show that, while one is not met by our grounding operator rules, the other is. We will then weaken the first criterion in a manner that suits, as we will argue, the nature of the considered grounding operators, and show that the weakened version is met by some of them, but not all. We will draw some conclusions concerning the operators and the relations that they formalise.

We consider now our first condition, often employed as criterion of the logicality of operators:  {\it deducibility of identicals} \cite{hac79}---also discussed in  \cite{pra71} under the name of {\it immediate expansion}.\footnote{A closely related condition employed as criterion of the logicality of operators is {\it uniqueness}, see  \cite{np15} for a detailed study of the two criteria and of the relations that they entertain. We decided not to consider {\it uniqueness} here for the simple fact that it trivially fails both for the mediate grounding operator and for the grounding tree operator. The reasons of the failure are simply that the introduction rules for these operators explicitly refer to occurrences of the operators themselves. While this failure might be of some interest with respect to the investigation of the proof-theoretical features  of inductively defined operators in general---indeed, inductive definitions in general essentially rely on the reference to other occurrences of the defined operator---it does not tell us very much about the differences in proof-theoretical behaviour between the grounding operators that we set out to study here.}


Let us first state the traditional, strict version of the criterion.
\bigskip

\noindent {\bf Deducibility of identicals} $\;$ An operator $\circ (\;\; , \ldots ,\;\; ) $  satisfies {\it deducibility of identicals} if, and only if, for any list of formulae $A_1 , \ldots ,A_n$, we can construct a derivation of $\circ (A_1 , \ldots ,A_n)$ from $\circ (A_1 , \ldots ,A_n)$
by applying an elimination rule for $\circ $ at least once, and by exclusively employing introduction and elimination rules for $\circ$.
\bigskip


In order to provide a positive example of employment of the condition, let us briefly exemplify how it  can be shown to hold for the traditional natural deduction  conjunction rules:\[\infer[\et i]{A\et B}{A&B}\qquad\qquad\qquad  \infer[\et e]{A}{A\et B}\qquad \infer[\et e]{B}{A\et B}\]Deducibility of identicals can be easily shown to hold for these rules by the following derivation:\[\infer[\et i]{A\et B}{\infer[\et e]{A}{A\et A}&\infer[\et e]{B}{A\et B}}\]

The condition, as can be seen from the previous example, requires the elimination rules for an operator to provide all the  information which is necessary in order to reintroduce the operator  itself. Notice moreover that it is essential that the rules for the operator under investigation alone are enough to show that the information obtained by eliminating it is sufficient to reintroduce it. In more general terms, this condition requires an {\it immediate schematic conformity} between the formulae that can be obtained by the elimination rules---which determine the ways we can use the operator---and the premisses of the introduction rules---which determine the truth conditions of the sentences constructed by applying the operator.  

That this strict version of the deducibility of identicals criterion is not met by  the immediate grounding operator is a rather obvious fact. There is no way, indeed, to introduce this operator without employing rules which are not rules for the operator itself: the introduction of the immediate grounding operator requires a grounding rule application. In other terms, there is no {\it immediate conformity} between the conclusion of the elimination rules and the premisses of the introductions rules.
 
The failure of deducibility of identicals for the grounding operator is not due to the form of the particular rules that we adopt for the operator. Indeed, even if we consider more direct rules to introduce the grounding operator, deducibility of identicals still fails. Consider indeed, for instance, the following grounding rule for conjunction:\[\infer={A\et B }{A&B}\]In our system, the relative grounding claim $A,B\groi A\et B $ would be introduced as follows:\[\infer{A,B\groi A\et B }{\infer={A\et B }{A&B}}\]But we could also define the following, more direct, rule:\[\infer{A,B\groi A\et B }{A & B & A\et B}\]
Nevertheless, also rules of this kind for introducing the grounding operator would fail the test related to deducibility of identicals because what we obtain by eliminating a generic instance of the grounding operator is not enough, in general, to infer a grounding claim. Indeed, the syntactic form of the formulae $G_1 , \ldots , G _n,  C$ is unknown and the following derivation could only be used to infer a grounding claim for certain specific choices of the formula $C$:\[\infer[?]{}{\infer{G_1 }{G_1 , \ldots , G _n  \groi C }&\ldots & \infer{G _n }{G_1 , \ldots , G _n  \groi C}& \infer{C}{G_1 , \ldots , G _n  \groi C} }\]In this case, then, there is no {\it schematic conformity} between the conclusion of the elimination rules and the premisses of the introductions rules.

The strict version of deducibility of identicals  is not met by the mediate grounding operator either. Indeed, in order to derive a claim of the form $\Gamma [\Delta ]\grom A$ by $\grom$ introductions, we need to derive either an immediate grounding claim $\Gamma [\Delta ]\groi A$ or two mediate grounding claims that yield $\Gamma [\Delta ]\grom A$  by transitivity. And we certainly cannot derive any of these claims from the hypothesis $\Gamma [\Delta ]\grom A$ by exclusively employing  $\grom$ elimination rules---or $\grom$ introduction rules, for that matter.
Finally, not even the grounding tree operator enjoys   deducibility of identicals. The only reason why this is the case, though, is that in the base case a grounding tree is an immediate grounding claim. And, as we have argued above, the immediate grounding operator does not meet the strict version of the criterion.

The failure of deducibility of identicals is not particularly surprising. Grounding operators, indeed, are not supposed to be purely logical operators---in fact, not even logical grounding operators are---because, in order to introduce them, the logical information that the premisses of the rule are derivable is not enough.\footnote{Here, by {\it logical information} we mean information exclusively concerning whether a formula is derivable from a certain set of hypotheses; as opposed, for instance, to information concerning the syntactic form of a formula, the particular way a formula is derivable from a set of hypotheses, or the semantical interpretation of the constituents of a formula.} This is an essential feature of these operators, because they do not only concern truth and deducibility but also a good dose of non-logical information. In the case of logical grounding, for instance, the syntactical complexity of formulae is an essential component of the conditions under which a grounding relation holds; and other notions of complexity, or of fundamentality, play a similar and essential role with respect to other notions of grounding. In general, while not all the constraints required to introduce  grounding operators can be explicitly expressed in the logical language---and hence encoded in the conclusion of the elimination rules---for different notions of grounding, we can guarantee that such constraints are met by proof-theoretical conditions on the derivations of the premisses of the introduction rules. This feature of the rules for grounding operators directly reflects the hyperintensional nature of grounding relations. Before discussing the generality of this connection between hyperintensionality and non-logicality, let us briefly clarify what we mean by hyperintensionality. 
%
%
A relation is hyperintensional if its terms cannot be  substituted {\it salva veritate} by logically equivalent ones in general. In other words, if $A$ is related to $B$ by a hyperintensional relation, the logical equivalence of $A$ and $A'$ is not enough to conclude that also $A'$ is related to $B$ by the same relation. We cannot claim here that hyperintensionality always implies non-logicality, since no general account of hyperintensionality in proof-theory exists yet. Nevertheless, there seem to be good grounds to argue that hyperintensionality do indeed imply a failure of the deducibility of identicals criterion, because the non-logical requirements of an operator that make it hyperintensional cannot be expressed by the purely logical information that can be conveyed through the conclusion of a rule. It is in any case indubitable that the particular reasons why the operators under consideration are non-logical are the same reasons that account for their hyperintensionality.

If we attribute to hyperintensionality the failure of grounding operators to meet the logicality criteria, it is natural to wonder whether the grounding operator rules really are unbalanced---as the failure of deducibility of identicals suggests---or simply present a weaker form of balance that makes them well-behaved rules as far as rules for hyperintensional operators are concerned. The property of having balanced sets of introduction and elimination rules, indeed, need not be a prerogative of logical connectives. It is desirable for sentential operators in general to have balanced rules, because having balanced rules simply means having rules that {\it exactly} characterise the behaviour of the operator. It means, that is,  that the rules for using the operator---i.e., its eliminations rules---enable us to use it exactly as specified by the rules that determine when it is true---i.e., its introduction rules. What distinguishes logical operators from other types of sentential operators is not the balance of their rules in itself, but that this balance holds relatively to the {\it kind} of information that we consider as legitimate to introduce them and that we expect them to yield when eliminated. In other terms, it is not the balance of its introduction and elimination rules alone that makes an operator logical, it is the fact that we can show them balanced by exclusively considering logical information, that is, information about the derivability of formulae\footnote{As mentioned above, by {\it logical information} we mean the information {\it that} certain formulae are derivable from certain sets of hypotheses. But not, for instance, information concerning {\it the way} certain formulae are derivable from certain sets of hypotheses.}. This particular kind of balance is exactly the one enforced by the {\it immediate schematic conformity} requirement implicit in the strict deducibility of identicals criterion, as already discussed.

In the following sections, we will precisely address the issue concerning the balance of the presented rules for grounding operators by investigating whether they admit detour reductions that allow for normalisation results, and whether they comply with a weaker version of the deducibility of identicals requirement.

\subsection{Detour eliminability}
\label{sec:intro-elim}

We will study now whether the rules for $\groi$, $\grom$ and $\grot$ meet a second condition, which we call {\it detour eliminability}. By {\it detour eliminability} we mean here that the application of the rules that govern the use of the operators does not yield more information than that required to apply the rules determining when sentences constructed by applying the operators are true. We thus have to show that the elimination rules for the grounding operators do not enable us to infer from grounding claims more than what is required to introduce them. In our case, this boils down to proving that suitable normalisation results can be proved for the calculi containing our rules.


Even though this condition is often considered as an essential requirement for several criteria of logicality of operators---see, for instance, \cite{pra77, dum91, rea00, ten07}---the existence of a set of rules for an operator that admits normalisation results is not always regarded as a sufficient condition for concluding that a given operator is a logical operator, see \cite{pog10} for a survey of the main existing accounts of logicality. Moreover, as we argued at the end of the previous section, the fact that a sentential operator is not logical does not imply that its rules must be ill-behaved in general and that it is of no interest to understand whether an exact definition of its meaning can be given by inferential means. We will, therefore, endeavour in the analysis of the behaviour of our grounding operator rules with respect to detour eliminability.

In order to show that our operators enjoy detour eliminability, we will define detour reductions similar in spirit to those employed in \cite{pra06} and we will show that the presented  reductions for $\groi$ and $\grot$ can be employed to generalise the normalisation result presented in \cite{gen21} for a grounding calculus based on the notion of logical grounding introduced in \cite{pog16}. We believe that the interest of this particular normalisation result is not limited to the notion of grounding on which the calculus presented in \cite{gen21} is based. Indeed, the normalisation strategy employed for proving it is a rather common and general one, and could very well apply to a variety of grounding calculi with similar proof-theoretical features. As far as the reductions for $\grom$ are concerned, on the other hand, we will discuss the problems that they pose with respect to normalisation, both from a technical and conceptual perspective. 

The reduction rules for $\groi$, $\grom$ and $\grot$ are presented in Tables \ref{tab:reductions-gro},  \ref{tab:reductions-tra}, \ref{tab:reductions-tree1} and \ref{tab:reductions-tree2}, respectively.

We first show that, if we extend the calculus in presented in \cite{gen21} by our rules for $\groi$ and $\grot$, then the normalisation result for it generalises. Afterwards, we discuss the problems arising with the rules for $\grom$.

\begin{definition}
Let us call $\GC$ the grounding calculus defined in \cite{gen21} and $\GC +$ the calculus defined by extending $\GC$ with all our rules for the immediate grounding and grounding tree operators.
\end{definition}

\begin{table}[h] \centering
\hrule \smallskip
\[\vcenter{\infer{A_i}{\infer{A_1 , \ldots , A_n [
C_1 , \ldots , C_m] \groi B}{\infer={B}{A_1 &
\ldots & A_n & [C_1 & \ldots & C_m] }}}} \; \red \; A_i
\qquad 
\vcenter{\infer{C_i}{\infer{A_1, \ldots , A_n [
C_1 , \ldots , C_m] \groi B}{\infer={B}{A_1 &
\ldots & A_n & [C_1 & \ldots & C_m]}}}} 
\; \red \; C_i \]

\[\vcenter{\infer{B}{\infer{A_1 , \ldots , A_n [C_1 , \ldots , C_m] \groi B}{\infer={B}{A_1 &
\ldots & A_n & [ C_1 & \ldots & C_m] }}}} \; \red \; \vcenter{\infer={B}{A_1 &
\ldots & A_n & [ C_1 & \ldots & C_m] }}
 \]\medskip

If $\quad \vcenter{ \infer{\fal}{A_1 , \ldots , A_n [
C_1 , \ldots ,C_m] \groi B}}\quad$ is an elimination of $\groi$, then $\groi$ cannot have been introduced since  $ \quad \vcenter{\infer={B}{A_1 &
\ldots & A_n & [ C_1 & \ldots & C_m] }} \quad $ is not a rule application\smallskip

\hrule \smallskip
\caption{Detour Reductions for $\groi$}\label{tab:reductions-gro}
\end{table}

\begin{table}[h] \centering

\hrule \smallskip
  \[
   \vcenter{\infer{\gamma }{\infer{\Gamma _1  , \Gamma  , \Gamma _2 [  \Delta _1 , \Delta   ]\grom   B }{\Gamma  [  \Delta]  \grom  A &&& \Gamma _1  , A , \Gamma _2 [  \Delta _1]  \grom   B}}} \; \red \; \vcenter{\infer{\gamma }{ \Gamma _1  , A , \Gamma _2 [  \Delta _1]  \grom   B}}\]where $\gamma \in \Gamma _1 , A , \Gamma _2 , \Delta _1, B$

  \[
   \vcenter{\infer{\gamma }{\infer{\Gamma _1  , \Gamma  , \Gamma _2 [  \Delta  , \Delta _1   ]\grom   B }{\Gamma  [  \Delta]  \grom  A &&& \Gamma _1  , A , \Gamma _2 [  \Delta _1]  \grom   B}}} \; \red \; \vcenter{\infer{\gamma }{ \Gamma  [  \Delta]  \grom  A}}\]where $\gamma \in \Gamma , \Delta , A$

  \[
   \vcenter{\infer{\gamma }{\infer{\Gamma _1 [  \Delta _1 , \Gamma , \Delta , \Delta _2]\grom   B }{\Gamma  [  \Delta]  \grom  A &&& \Gamma _1   [  \Delta _1 , A, \Delta _2]  \grom   B}}} \; \red \; \vcenter{\infer{\gamma }{ \Gamma _1   [  \Delta _1 , A, \Delta _2]  \grom   B}}\]where $\gamma \in \Gamma _1, \Delta _1 , A, \Delta _2  , B$

  \[
   \vcenter{\infer{\gamma }{\infer{\Gamma _1 [  \Delta _1 , \Gamma , \Delta , \Delta _2]\grom   B }{\Gamma  [  \Delta]  \grom  A &&& \Gamma _1   [  \Delta _1 , A, \Delta _2]  \grom   B}}} \; \red \; \vcenter{\infer{\gamma }{ \Gamma  [  \Delta]  \grom  A}}\]where $\gamma \in \Gamma , \Delta , A$

  \[
   \vcenter{\infer{\gamma}{\infer{\Gamma  [  \Delta  ]\grom   A }{\Gamma  [  \Delta]  \groi   A} }} \; \red \; \vcenter{\infer{\gamma }{ \Gamma  [  \Delta]  \groi  A}}\]where $\gamma \in \Gamma , \Delta , A $

\medskip
  \hrule \smallskip
\caption{Detour Reductions for $\grom$}\label{tab:reductions-tra}
\end{table}

\begin{table}[h] \centering
\hrule \smallskip
  \[ 
   \vcenter{\infer{\Gamma  [\Delta ] \groi A_i }{   \infer{A_1 , \ldots ,( \Gamma  [\Delta ] \groi A_i ) , \ldots , A_n  [ C] \groi
B}{\Gamma  [\Delta ] \groi A_i &&& A_1 , \ldots , A_i  , \ldots , A_n [\Sigma ] \groi B}}} \; \red \; \Gamma  [\Delta ] \groi A_i \]\smallskip
\[\vcenter{\infer{\Gamma  [\Delta ] \groi C_i }{  \infer{\Sigma  [  C_1 , \ldots , (\Gamma  [\Delta ] \groi C_i ) , \ldots , C_n]  \groi
B}{ \Gamma  [\Delta ] \groi C_i &&& \Sigma  [ C_1 , \ldots , C_i , \ldots , C_n]  \groi  B }}} \; \red \; \Gamma  [\Delta ] \groi C_i\]\smallskip
 \[\vcenter{\infer{A_1 , \ldots , A_i  , \ldots , A_n [  \Sigma ] \groi B }{    \infer{A_1 , \ldots ,( \Gamma  [\Delta ] \groi A_i ) , \ldots , A_n  [  \Sigma ] \groi
B}{\Gamma  [\Delta ] \groi A_i &&& A_1 , \ldots , A_i  , \ldots , A_n [\Sigma ] \groi B}}} \; \red \; A_1 , \ldots , A_i  , \ldots , A_n [  \Sigma ] \groi B \]\smallskip
\[\vcenter{\infer{\Sigma  [C_1 , \ldots , C_i , \ldots , C_n]  \groi  B }{    \infer{\Sigma  [  C_1 , \ldots , (\Gamma  [\Delta ] \groi C_i ) , \ldots , C_n]  \groi
B}{ \Gamma  [\Delta ] \groi C_i &&& \Sigma  [ C_1 , \ldots , C_i , \ldots , C_n]  \groi  B }}} \; \red \; \Sigma  [C_1 , \ldots , C_i , \ldots , C_n]  \groi  B\]
\hrule \smallskip
\caption{Detour Reductions for $\grot$, part 1}\label{tab:reductions-tree1}
\end{table}

\begin{table}[h] \centering

  \hrule \smallskip
%
%
%
{\small 

\[ 
   \vcenter{\infer{\Xi  [\Theta ] \groi D }{   \infer{A_1 , \ldots ,( \Gamma  [\Delta ] \groi A_i ) , \ldots , A_n  [ C] \groi
B}{\Gamma  [\Delta ] \groi A_i &&& A_1 , \ldots , A_i  , \ldots , A_n [\Sigma ] \groi B}}} \; \red \; \vcenter{\infer{\Xi  [\Theta ] \groi D}{\Gamma  [\Delta ] \groi A_i}} \quad \text { where } (\Xi  [\Theta ] \groi D ) \in \Gamma \cup \Delta\]

\[ 
   \vcenter{\infer{\Xi  [\Theta ] \groi D }{   \infer{A_1 , \ldots ,( \Gamma  [\Delta ] \groi A_i ) , \ldots , A_n  [ C] \groi
B}{\Gamma  [\Delta ] \groi A_i &&& A_1 , \ldots , A_i  , \ldots , A_n [\Sigma ] \groi B}}} \; \red \; \vcenter{\infer{\Xi  [\Theta ] \groi D}{A_1 , \ldots , A_i  , \ldots , A_n [\Sigma ] \groi B}} \]where $ (\Xi  [\Theta ] \groi D ) \in \{A_1 , \ldots , A_n\} \cup \Sigma$

\[\vcenter{\infer{\Xi  [\Theta ] \groi D }{  \infer{\Sigma  [  C_1 , \ldots , (\Gamma  [\Delta ] \groi C_i ) , \ldots , C_n]  \groi
B}{ \Gamma  [\Delta ] \groi C_i &&& \Sigma  [ C_1 , \ldots , C_i , \ldots , C_n]  \groi  B }}} \; \red \; \vcenter{\infer{\Xi  [\Theta ] \groi D}{\Gamma  [\Delta ] \groi C_i}} \quad \text { where } (\Xi  [\Theta ] \groi D ) \in \Gamma \cup \Delta\]

\[\vcenter{\infer{\Xi  [\Theta ] \groi D }{  \infer{\Sigma  [  C_1 , \ldots , (\Gamma  [\Delta ] \groi C_i ) , \ldots , C_n]  \groi
B}{ \Gamma  [\Delta ] \groi C_i &&& \Sigma  [ C_1 , \ldots , C_i , \ldots , C_n]  \groi  B }}} \; \red \; \vcenter{\infer{\Xi  [\Theta ] \groi D}{\Sigma  [ C_1 , \ldots , C_i , \ldots , C_n]  \groi  B}}  \]where $ (\Xi  [\Theta ] \groi D ) \in  \Sigma\cup \{C_1 , \ldots , C_n\}$}

%
%
\medskip\hrule \smallskip
\caption{Detour Reductions for $\grot$, part 2}\label{tab:reductions-tree2}
\end{table}

We recall the definition of reduction of a derivation and some related terminology.

\begin{definition}[Reductions, Redexes and Critical Rules] For any four derivations
$s, s', d$ and $d'$, if $s\red s'$, $d$
contains $s$ as a subderivation, and $d'$ can be obtained by replacing
$s$ with $s'$ in $d$, then the relation $d\red d'$ holds and we say
that $d$ reduces to $d'$.

We denote by $\red ^* $ the reflexive and transitive closure of
$\red$.

As usual, if the bottommost rule of a derivation---or subderivation---$d$ and one of the
rules applied immediately above it form a configuration 
shown to the left of $\red$ in Tables \ref{tab:reductions-gro},  \ref{tab:reductions-tra}, \ref{tab:reductions-tree1} and \ref{tab:reductions-tree2}, then we say that $d$ is \emph{a redex}. We call the last two rule applications of a redex the
\emph{critical rules} of the redex.
\end{definition}

We provide some simple and usual definitions that will be used in the normalisation proof. 
\begin{definition}[Logical Complexity] The logical complexity of a formula is defined as usual by counting the number of symbols in the formula.
\end{definition}

\begin{definition}[Redex Complexity] The complexity of a redex $r$ is
defined as the logical complexity of the formula
introduced by the uppermost critical rule of $r$.
\end{definition}

\begin{definition}[Normal Form] We say that a $\GC +$ derivation $d$ is
normal, or in normal form, if there is no derivation $d'$ such that $d
\red d'$ holds.
\end{definition} Clearly, being normal and not containing any redex are equivalent properties.

The normalisation will follow the ideas employed in~\cite{ts96}.  The main intuition  behind the proof is that, generally, by applying a reduction rule, we eliminate a redex of a certain complexity and possibly generate new redexes of smaller complexity. For most reductions this is all that matters. If we reduce a suitable redex in the derivation, we either have a decrease of the maximal complexity of redexes, or a decrease fo the number of redexes with maximal complexity. If all our reduction rules allowed for such an argument, we could prove the termination of the normalisation procedure by induction on a pair of values corresponding to the maximal complexity of the redexes in the derivation and the number of redexes with maximal complexity occurring in the derivation. Nevertheless, not all reduction rules are this well-behaved, indeed some of them implement permutations between rules, and a permutation does not change the complexity of redexes. Hence, we need a method to keep track of permutations and to account for them in our complexity measure. In order to do so, we adapt the notion of  {\emph segment} defined in \cite{ts96}. A segment is a path inside the derivation tree which connects two rule applications and satisfies the following  two conditions: first, it connects two rule applications that would form a redex if they occurred one immediately after the other---and the redex must be different from a  permutation redex; second, it can be shortened by using permutations in order to eventually obtain the redex formed by the two rule applications.

It is easy to see, by inspection of the proof in \cite{gen21}, that all definitions and permutations  generalise if we treat the introduction rules for $\groi$ and $\grot$ as any other  introduction rule, and the elimination rules for $\groi$ and $\grot$ as any other  elimination rule. Intuitively, $\groi$ and $\grot$ elimination rules and redexes behave very similarly to $\et$ elimination rules and redexes, and the differences in their introduction rules do not generate any particular problem. 

We recall the definitions introduced in \cite{gen21}. 
%

\begin{definition}[Segment (adapted from Def.~6.1.1.~in~\cite{ts96})
and Segment Complexity]  For any $\GC +$ derivation $d$, a segment of
length $n$ in $d$ is a sequence $A_1, \ldots , A_n$ of formula
occurrences in $d$ such that the following holds.
\begin{enumerate}
\item For $1 < i < n$, one of the following holds:
\begin{itemize}
\item $A_i$ is a minor premiss of an
application of $\vel $ elimination in $d$ with conclusion
$A_{i+1}=A_i$, 
\item $A_i$ is the premiss of a non-logical rule such that its conclusion  $A_{i+1}$ has the same  logical complexity as $A_i$ (for the calculus in \cite{gen21}, these rules are all $\ace$ rules and those converse rules that do not induce a change of logical complexity from premiss to conclusion).
\end{itemize}
 
 \item $A_n$ is neither the minor premiss of a $\vel $ elimination, nor the premiss of a  non-logical  rule the conclusion of which has the
same logical complexity as $A_n$ (for the calculus in \cite{gen21}, these rules are all $\ace$ rules and those converse rules that do not induce a change of logical complexity from premiss to conclusion).
 
 \item $A_1$ is neither the conclusion of a $\vel $ elimination, nor the conclusion of a non-logical rule the premiss of which has the same logical complexity as $A_1$ (for the calculus in \cite{gen21}, these rules are all $\ace$ rules and those converse rules that do not induce a change of logical complexity from premiss to conclusion).
\end{enumerate} For any segment, if
\begin{itemize}
\item $n>0$ or
\item $A_1$ is the conclusion of an introduction rule and $A_n$ is the
major premiss of an elimination rule
\end{itemize}then the complexity of the segment is the logical
complexity of $A$. Otherwise, the  complexity of the segment is $0$.
\end{definition}
Notice that all formulae in a segment have the same logical complexity. This is obvious for the case of $\vel$ eliminations and it holds by assumption for the other cases.

We introduce some terminology to describe the relative position of two segments in a derivation and prove a simple fact about the arrangement  of segments in a derivation which will be used in the normalisation proof.
\begin{definition}[Terminology for Segments] If a segment contains
only one formula occurrence, by \emph{reducing the segment} we mean
reducing---if possible---the non-permutation redex the critical rules
of which are applied immediately above and immediately below the
formula; otherwise, we mean reducing the permutation redex of the
$\vel$ elimination which has the bottommost formula of the segment as conclusion.

A segment $r$ \emph{occurs above} a segment $s$ if the bottommost
formula of $r$ occurs above the bottommost formula of $s$.

A segment $r$ \emph{occurs to the right} of a segment $s$ if there are
derivations $\rho$ and $\sigma $ such that some formula of $r$ occurs
in $\rho$, some formula of $s$ occurs in $\sigma$, the root of $\rho$
and the root of $\sigma $ are premisses of the same rule application,
and the root of $\rho$ occurs to the right of the root of $\sigma $
with respect to such rule application.
\end{definition}

\begin{proposition}\label{prop:segment-selection} For any two distinct
segments in a derivation $d$, if neither is to the right of the other,
then one is above the other.
\end{proposition}

\begin{proof} See \cite{gen21}.
\end{proof}

We prove normalisation for the calculus presented in \cite{gen21} extended by the rules for $\groi$ and  $\grot$.
\begin{theorem} \label{thm:norm} For any derivation $d$, there is a derivation
$d'$ such that $d$ can be reduced to $d'$ in a finite number of
reductions and $d'$ is normal.
\end{theorem}
\begin{proof} We employ the following reduction strategy. We reduce
any rightmost segment of maximal complexity that does not occur below
any other segment of maximal complexity. 
By Proposition~\ref{prop:segment-selection}, we can always find such a segment.

We prove that this reduction strategy always produces a series of
reductions which is of finite length and which results in a normal
form.

We define the complexity of a derivation $d$ to be the triple of
natural numbers $(m,n,u)$, where $m$ is the complexity of the segments
in $d$ with maximal segment complexity, $n$ is the sum of the lengths
of the segments in $d$ with segment complexity $m$, and $u$ is the
number of rule applications in $d$. We then fix a generic derivation
$d$ and reason by induction on the lexicographic order on triples of
natural numbers.

If the complexity of $d$ is $(0,0, u)$ then $d$ is normal and the
claim holds.

Suppose now that the complexity of $d$ is $(m,n,u)$, that $m+n>0$, and
that for each derivation simpler than $d$ the claim holds. Since
$m+n>0$, there must be at least one maximal segment in $d$ that does
not occur below any other maximal segment. We reduce one of such
segments and reason by cases on the
shape of the reduction. We only consider some exemplar cases, other cases can be found in \cite{gen21}.
\begin{itemize}
\item \[\vcenter{\infer{A_i}{\infer{A_1 , \ldots , A_n [
C_1,\ldots ,C_m] \groi B}{\infer={B}{A_1 &
\ldots & A_n & [C_1&\ldots &C_m]}}}} \; \red \; A_i\]
By the
reduction we eliminate one maximal segment. We show now that no
segment of maximal complexity has been duplicated, the length of no
segment of maximal complexity has been increased, and no segment has
become as complex as the reduced one; and hence that the complexity of
$d'$ is $(m',n',u')<(m,n,u)$ since we either reduced the maximal
complexity of the segments or the sum of the lengths of the segments
with maximal complexity. For each segment in $d$ exactly one of the
following holds: \uno the segment does not contain any of the
displayed occurrences of $A_i$ and $A_1 , \ldots , A_n [C_1,\ldots ,C_m] \groi B$, \due the segment contains
some of the displayed occurrences of $A_i$, \tre the segment contains
the displayed occurrence of $A_1 , \ldots , A_n [C_1,\ldots ,C_m] \groi B$. If \uno the segment has neither
been modified nor been duplicated by the reduction. If \due the
reduction might join the segment with another one for which \due
holds, but the resulting segment is still less complex than the
reduced one since $A$ is less complex than $A_1 , \ldots , A_n [C_1,\ldots ,C_m] \groi B$. We just  eliminated the only segment for which \tre holds.

\item \[ \vcenter{\infer{C_i}{\infer{A_1, \ldots , A_n [C_1,\ldots ,C_m] \groi B}{\infer={B}{A_1 &
\ldots & A_n & [C_1&\ldots &C_m]}}}} \; \red \; C_i \]
By the
reduction we eliminate one maximal segment. We show now that no
segment of maximal complexity has been duplicated, the length of no
segment of maximal complexity has been increased, and no segment has
become as complex as the reduced one; and hence that the complexity of
$d'$ is $(m',n',u')<(m,n,u)$ since we either reduced the maximal
complexity of the segments or the sum of the lengths of the segments
with maximal complexity. For each segment in $d$ exactly one of the
following holds: \uno the segment does not contain any of the
displayed occurrences of $C_i$ and $A_1 , \ldots , A_n [C_1,\ldots ,C_m] \groi B$, \due the segment contains
some of the displayed occurrences of $C_i$, \tre the segment contains
the displayed occurrence of $A_1 , \ldots , A_n [C_1,\ldots ,C_m] \groi B$. If \uno the segment has neither
been modified nor been duplicated by the reduction. If \due the
reduction might join the segment with another one for which \due
holds, but the resulting segment is still less complex than the
reduced one since $C_i$ is less complex than $A_1 , \ldots , A_n [C_1,\ldots ,C_m] \groi B$. We just  eliminated the only segment for which \tre holds.

\item \[\vcenter{\infer{B}{\infer{A_1 , \ldots , A_n [C_1,\ldots ,C_m] \groi B}{\infer={B}{A_1 &
\ldots & A_n & [C_1&\ldots &C_m] }}}} \; \red \; \vcenter{\infer={B}{A_1 &
\ldots & A_n & [C_1&\ldots &C_m] }}
 \]By the
reduction we eliminate one maximal segment. We show now that no
segment of maximal complexity has been duplicated, the length of no
segment of maximal complexity has been increased, and no segment has
become as complex as the reduced one; and hence that the complexity of
$d'$ is $(m',n',u')<(m,n,u)$ since we either reduced the maximal
complexity of the segments or the sum of the lengths of the segments
with maximal complexity. For each segment in $d$ exactly one of the
following holds: \uno the segment does not contain any of the
displayed occurrences of $B$ and $A_1 , \ldots , A_n [C_1,\ldots ,C_m] \groi B$, \due the segment contains
some of the displayed occurrences of $B$, \tre the segment contains
the displayed occurrence of $A_1 , \ldots , A_n [C_1,\ldots ,C_m] \groi B$. If \uno the segment has neither
been modified nor been duplicated by the reduction. If \due the
reduction might join the segment with another one for which \due
holds, but the resulting segment is still less complex than the
reduced one since $B$ is less complex than $A_1 , \ldots , A_n [C_1,\ldots ,C_m] \groi B$. We just  eliminated the only segment for which \tre holds.

\item {\small  \[    \vcenter{\infer{\Gamma  [\Delta ] \groi A_i }{   \infer{A_1 , \ldots ,( \Gamma  [\Delta ] \groi A_i ) , \ldots , A_n  [ \Sigma] \groi
B}{\Gamma  [\Delta ] \groi A_i &&& A_1 , \ldots , A_i  , \ldots , A_n [\Sigma ] \groi B}}} \; \red \; \Gamma  [\Delta ] \groi A_i \]
}By the
reduction we eliminate one maximal segment. We show now that no
segment of maximal complexity has been duplicated, the length of no
segment of maximal complexity has been increased, and no segment has
become as complex as the reduced one; and hence that the complexity of
$d'$ is $(m',n',u')<(m,n,u)$ since we either reduced the maximal
complexity of the segments or the sum of the lengths of the segments
with maximal complexity. For each segment in $d$ exactly one of the
following holds: \uno the segment does not contain any of the
displayed occurrences of $\Gamma  [\Delta ] \groi A_i$ and the displayed occurrence of  $A_1 , \ldots ,( \Gamma  [\Delta ] \groi A_i ) , \ldots , A_n [  \Sigma] \groi
B$, \due the segment contains
some of the displayed occurrences of $\Gamma  [\Delta ] \groi A_i$, \tre the segment contains
the displayed occurrence of $A_1 , \ldots ,( \Gamma  [\Delta ] \groi A_i ) , \ldots , A_n [\Sigma ] \groi
B$. If \uno the segment has neither
been modified nor been duplicated by the reduction. If \due the
reduction might join the segment with another one for which \due
holds, but the resulting segment is still less complex than the
reduced one since $\Gamma  [\Delta ] \groi A_i$ is less complex than $A_1 , \ldots ,( \Gamma  [\Delta ] \groi A_i ) , \ldots , A_n [\Sigma] \groi
B$. We just  eliminated the only segment for which \tre holds.

\item {\small \[\vcenter{\infer{\Sigma  [C_1 , \ldots , C_i , \ldots , C_n]  \groi  B }{    \infer{\Sigma  [  C_1 , \ldots , (\Gamma  [\Delta ] \groi C_i ) , \ldots , C_n]  \groi
B}{ \Gamma  [\Delta ] \groi C_i &&& \Sigma  [ C_1 , \ldots , C_i , \ldots , C_n]  \groi  B }}} \; \red \; \Sigma  [C_1 , \ldots , C_i , \ldots , C_n]  \groi  B\]}By the
reduction we eliminate one maximal segment. We show now that no
segment of maximal complexity has been duplicated, the length of no
segment of maximal complexity has been increased, and no segment has
become as complex as the reduced one; and hence that the complexity of
$d'$ is $(m',n',u')<(m,n,u)$ since we either reduced the maximal
complexity of the segments or the sum of the lengths of the segments
with maximal complexity. For each segment in $d$ exactly one of the
following holds: \uno the segment does not contain any of the
displayed occurrences of $\Sigma  [C_1 , \ldots , C_i , \ldots , C_n]  \groi  B$ and the displayed occurrence of  $\Sigma  [  C_1 , \ldots , (\Gamma  [\Delta ] \groi C_i ) , \ldots , C_n$, \due the segment contains
some of the displayed occurrences of $\Sigma  [C_1 , \ldots , C_i , \ldots , C_n]  \groi  B$ , \tre the segment contains
the displayed occurrence of $\Sigma  [  C_1 , \ldots , (\Gamma  [\Delta ] \groi C_i ) , \ldots , C_n$. If \uno the segment has neither
been modified nor been duplicated by the reduction. If \due the
reduction might join the segment with another one for which \due
holds, but the resulting segment is still less complex than the
reduced one since $\Sigma  [C_1 , \ldots , C_i , \ldots , C_n]  \groi  B$ is less complex than $\Sigma  [  C_1 , \ldots , (\Gamma  [\Delta ] \groi C_i ) , \ldots , C_n$. We just  eliminated the only segment for which \tre holds.

\item {\small \[ 
   \vcenter{\infer{\Xi  [\Theta ] \groi D }{   \infer{A_1 , \ldots ,( \Gamma  [\Delta ] \groi A_i ) , \ldots , A_n  [ C] \groi
B}{\Gamma  [\Delta ] \groi A_i &&& A_1 , \ldots , A_i  , \ldots , A_n [\Sigma ] \groi B}}} \; \red \; \vcenter{\infer{\Xi  [\Theta ] \groi D}{\Gamma  [\Delta ] \groi A_i}} \]}where $ (\Xi  [\Theta ] \groi D ) \in \Gamma \cup \Delta$.
By the reduction we eliminate one maximal segment. We show now that no segment of maximal complexity has been duplicated, the length of no segment of maximal complexity has been increased, and no segment has become as complex as the reduced one; and hence that the complexity of
$d'$ is $(m',n',u')<(m,n,u)$ since we either reduced the maximal
complexity of the segments or the sum of the lengths of the segments with maximal complexity. For each segment in $d$ exactly one of the
following holds: \uno the segment does not contain  the  displayed occurrence of $\Gamma  [\Delta ] \groi A_i$, the displayed occurrence of  $A_1 , \ldots ,( \Gamma  [\Delta ] \groi A_i ) , \ldots , A_n [  \Sigma] \groi
B$ and the displayed occurrence of $\Xi  [\Theta ] \groi D $; \due the segment contains the displayed occurrence of $\Gamma  [\Delta ] \groi A_i$; \tre the segment contains the displayed occurrence of $\Xi  [\Theta ] \groi D$;  \quattro the segment contains
the displayed occurrence of $A_1 , \ldots ,( \Gamma  [\Delta ] \groi A_i ) , \ldots , A_n [\Sigma ] \groi
B$. If \uno the segment has neither
been modified nor been duplicated by the reduction. If \due or \tre the
reduction might join the segment with another one for which \due or \tre 
holds, but the resulting segments are less complex than the
reduced one since both $\Gamma  [\Delta ] \groi A_i$ and $\Xi  [\Theta ] \groi D$---which is a subformula of $\Gamma  [\Delta ] \groi A_i$---are  less complex than $A_1 , \ldots ,( \Gamma  [\Delta ] \groi A_i ) , \ldots , A_n [\Sigma] \groi
B$. We just  eliminated the only segment for which \tre holds.\end{itemize}
\end{proof}

\subsubsection{Mediate grounding and global detour eliminability}

Now that we have shown that the detour reductions for $\groi$ and $\grot$ enable us to generalise the normalisation result in \cite{gen21},
we consider the reductions of the detours generated by the rules for $\grom$. Since each individual detour generated by rules for $\groi,\grot$ and $\grom$ can be suitably reduced---this is obvious if we consider the reduction rules in Tables \ref{tab:reductions-gro},  \ref{tab:reductions-tra}, \ref{tab:reductions-tree1} and \ref{tab:reductions-tree2}---we can state that all three operators enjoy a local form of detour eliminability: the information that we can obtain from the elimination of a grounding operator $\circ$ occurring in a grounding claim $\Gamma [\Delta ]\circ A$, does not exceed the information required to derive the grounding claim $\Gamma [\Delta ]\circ A$ by introducing the operator $\circ$. But while the detour reductions for $\groi$ and $\grot$ reduce the logical complexity of the formulae occurring in the considered derivation in a rather usual way---and it is hence possible to show global detour eliminability results for $\groi$ and $\grot$ by using standard techniques---the detour reductions for $\grom$ do not, in general, only generate detours of smaller logical complexity.  This is due to the very peculiar fact that the introduction rules for the $\grom $ operator---as opposed to the introduction rules for $\groi$ and $\grot$---might not have premisses which are simpler than their conclusion. This can happen with derivable grounding claims if our underlying grounding calculus captures a non-logical notion of grounding. For instance, if $p\grom q $ and $q\grom r$ are not contradictory grounding claims according to our notion of grounding, then\[\infer{p\grom r}{p\grom q &q\grom r}\]is a perfectly legitimate derivation by $\grom$ introduction of a true mediate grounding claim under the assumption that $p\grom q $ and $q\grom r$  are true grounding claims. And here nothing tells us, from a proof-theoretical perspective, that $p\grom r$ is more complex than $p\grom q $ and $q\grom r$. Similar problems, nevertheless, can occur even if our underlying grounding calculus captures a notion of logical grounding. Indeed, we cannot exclude in general the possibility that contradictory grounding claims occur in derivations---otherwise it would be impossible, for example, to show that certain grounding claims are false or contradictory. Hence, for instance, the following configuration might certainly occur in a logical grounding  derivation:\[\infer{p, t \grom u}{p \grom q\vel r\vel s & q\vel r\vel s ,t \grom u}\]As mentioned above, the decrease of logical complexity possibly induced by $\grom $ introduction rules implies, in turn,  that some detour reductions generate detours of greater logical complexity. As we can see here:
{\small \[\vcenter{\infer{p}{\infer{p, t \grom u}{\infer{p \grom q\vel r\vel s}{p\grom v\et z&v\et z\grom q\vel r\vel s } & q\vel r\vel s ,t \grom u}}}\quad \red \quad\vcenter{\infer{p}{\infer{p \grom q\vel r\vel s}{p\grom v\et z&v\et z\grom q\vel r\vel s }}}\]}where we eliminate a detour the complexity of which is the complexity of the formula $p,t\grom u$ and generate a detour the complexity of which is the complexity of $p\grom q\vel r\vel s$. Clearly, the decrease in logical complexity that might be induced by a $\grom$ elimination is closely related to the fact that combining two grounding claims by transitivity often involves  a loss of information. 

From a technical perspective, in conclusion, a general termination argument for the normalisation of calculi containing $\grom$ rules based on their schematic form seems very problematic. This means, in turn, that there is no clear way to show, through a normalisation termination argument, that $\grom$ enjoys global detour eliminability with respect to a generic grounding calculus. 
%
%

Different methods to show global detour eliminability results---by a termination proof for the normalisation procedure---for $\grom$ seem, nevertheless, possible if we consider as legitimate the option of employing the intended meaning of a mediate grounding statement and, in particular, by exploiting the correspondence between each statement of this kind and a grounding tree. Indeed, even though the conclusion of a $\grom $ introduction rule is not necessarily more complex than its premisses---and this is the reason why the reduction of detours does not decrease the complexity of a derivation in a standard sense---the conclusion of an application of the $\grom$ introduction rule always corresponds to a larger grounding tree than the premisses. This is the case because connecting two grounding claims by transitivity exactly corresponds to replacing a leaf of a grounding tree by another grounding tree. It seems therefore possible to use a complexity measure based on this correspondence to prove that also derivations containing $\grom$ detours can be normalised. It is not clear though whether such a complexity measure interacts well with the logical complexity used to show that the other sequences of detour  reductions terminate.  Notice moreover that such a complexity measure would not be a syntactic one, but a semantical one, and indeed if a premiss of the $\grom$ introduction rule is an incorrect grounding statement that violates the complexity constraints of grounding---such as, for instance,  $p\et q \et r \et s \grom p$ with respect to most logical grounding notions---then the complexity measure based on the corresponding grounding tree is undefined, since there is no corresponding grouding tree. 
A further requirement to the successful application of this technique in a termination proof of the reductions of $\grom$ detours is the possibility of determining, for each mediate grounding claim, the corresponding grounding tree by simple inspection of the mediate grounding claim itself. And while this is feasible for most logical grounding notions---it is indeed easy to reconstruct the grounding tree corresponding to any legitimate mediate logical grounding statement---for more complex notions of grounding, the inspection of a mediate grounding statement could not be enough to determine the tree structure it refers to---in particular if the transitive closure of the underlying immediate grounding relation is not decidable. Definitive technical results in this direction, though, must be obviously left to investigations concerning individual calculi that capture specific notions of grounding. 

A further discussion of the connections between decidability, or undecidability, of grounding relations and the behaviour of the relative notion of mediate grounding is postponed to Section \ref{sec:elim-intro} since these connections will play a key role also in that section.

\subsection{Weak deducibility of identicals}
\label{sec:elim-intro}

At the beginning of Section \ref{sec:balance}, we have shown that deducibility of identicals does not hold for our grounding operators and we have put  this failure in relation with the hyperintensionality of grounding and with the fact that grounding operators are not, strictly speaking, logical operators. Afterwards,  we have shown that, even though our rules for grounding operators would not pass a logicality test, they still suitably define both the immediate grounding operator and the grounding tree operator insofar as they enjoy detour eliminability. One might wonder then whether our rules for grounding operators enjoy some kind of complete balance even though they cannot be taken to define logical operators. We, therefore, endeavour in the definition of a weaker version of the deducibility of identicals criterion which determines in what sense our introduction and elimination rules for grounding operators are balanced, and that might prove of use with respect to the rules for hyperintensional operators in general.

A further reason to define a subtler balance  criterion for grounding operator exists, and it is related to the fact that deducibility of identicals fails in different ways for the three grounding operators. The failure of the deducibility of identicals condition for the grounding tree operator, indeed, essentially  depends on the failure of this criterion for immediate grounding. The failure of the condition for the mediate grounding operator, on the other hand, is complete and independent with respect to the proof-theoretical features of the immediate grounding operator. This suggests that a subtler criterion might enable us to better understand where the problem lies and possibly to distinguish between a partial failure of the deducibility of identicals criterion---that relative to immediate grounding and grounding trees---and a severer failure---that relative to mediate grounding. This might, moreover, further enlighten the reasons of the differences in the behaviour of the three operators that we have encountered in studying detour eliminability.

We define, hence,  a {\it weak deducibility of identicals} criterion by taking into account that the introduction rules of our immediate grounding operator essentially refer to other rules---that is, the grounding rules of the chosen underlying calculus. This means, in some sense, that in defining the criterion we attribute the due importance to the fact that our operators are not, strictly speaking, logical operators, because they are hyperintensional. Indeed, the hyperintensionality of grounding is essentially related to the fact that valid grounding claims depend on a non-logical hierarchy---for instance, the hierarchy induced by syntactic complexity for logical grounding, or metaphysical fundamentality for metaphysical grounding. This hierarchy can be internalised in a grounding calculus by restricting the form of its grounding rule schemata. The weak deducibility of identicals criterion that we will introduce, then, could be seen as a relativisation of the strict version of this criterion to the non-logical hierarchy on which grounding is based via the relativisation of the criterion to the set of grounding rules contained in the considered grounding calculus.

The criterion that we will present tells us that the information which  we can obtain by eliminating an occurrence of an operator contains all the information required to reintroduce the occurrence of the operator. This is exactly what the strict deducibility of identical criterion tells us about an operator; the only, yet essential, difference between the two is that the weak version of the criterion enables us to take into consideration and use---in order to show that this balance between introduction and elimination rules for the considered operator holds---the rules of our background calculus, to which the introduction rules of the operator refer. This reflects the idea that the introduction rules under study essentially rely---as they are are supposed to do---on non-logical information encoded in the rules of the underlying grounding calculus itself. This information is conveyed, in particular, by the specific rules that have been applied to derive the premisses of our introduction rule. This information is, therefore, not only about  {\it the fact that} the premisses have been derived, but also about {\it the way} they have been derived. Since the additional information that we consider in weakening the criterion is non-logical, the criterion does not fare well as a logicality criterion. Nevertheless, the weak criterion still constitutes an indication that a balance between introduction and elimination rules for the operator exists, even though the operator is not, strictly speaking, a logical one and hence the balance is not based on {\it immediate schematic conformity} relations between these rules.

\bigskip

\noindent {\bf Weak deducibility of identicals} $\; $ An operator $\circ (\;\; , \ldots ,\;\; ) $  satisfies {\it weak deducibility of identicals} with respect to a calculus $\kappa$ if, and only if, for any list of formulae $A_1 , \ldots ,A_n $ such that $\circ ( A_1 , \ldots , A_n)$ can be the conclusion of a $\circ $ introduction application, 
we can construct a derivation from $\circ ( A_1 , \ldots , A_n)$ to $\circ (A_1 , \ldots , A_n)$ by applying an elimination rule for $\circ$ at least once, and by exclusively employing rules for $\circ$ or rules which are  explicitly mentioned in the applicability conditions on the $\circ$ introduction rule.
\bigskip

The criterion, in other words, requires that, if the outermost occurrence of $\circ$ in $\circ ( A_1 , \ldots , A_n)$ can be introduced at all, then the logical information provided by eliminating  it and the non-logical information contained in the definition of the $\circ$ introduction rules is sufficient to reintroduce this occurrence of $\circ$.
The requirement that there must exists a $\circ $ introduction application with conclusion $\circ ( A_1 , \ldots , A_n)$ is needed here because the hyperintensional nature of grounding operators requires us to impose particularly strict conditions on their introduction rules, and thus there can exist a formula $A$ with a grounding operator as outermost connective such that no introduction rule application can have $A$ as conclusion.\footnote{If we compare grounding operators with extensional operators---such as conjunction and disjunction in both classical and intuitionistic logic---or intensional operators---such as intuitionistic implication and the necessity operator of most modal logics---we will see that it is possible to define introduction rules for the extensional connectives that require no conditions on the derivations of their premisses, and introduction rules for the intensional ones that only require conditions on the hypotheses employed to derive their premisses. Most introduction rules for extensional operators, then, can always be applied, regardless of how their premisses are derived. And while it might be the case that an introduction rule for an intensional operator $\circ$ cannot be applied to formulae derived from certain hypotheses; it is usually the case that, for any formula $\circ (A_1 , \ldots , A_n)$, it is possible to find a set of hypotheses that enable us to derive $\circ (A_1 , \ldots , A_n)$ by $\circ $ introduction. If we consider the introduction rules for grounding operators, this is not always possible. Indeed, due to the conditions on these rules, certain grounding claims will never be the conclusion of a legitimate introduction rule application.} This is a byproduct of the fact that, in order to define correct rules for the grounding operators, we also need to impose conditions on {\it the way} the premisses are derived, and not only on their derivability.

We show now that our immediate grounding operator $\groi$ and our operator for grounding trees meet the weak deducibility of identicals criterion.

\begin{proposition}\label{prop:wdoi-gro}
The immediate grounding operator $\groi$ enjoys weak deducibility of identicals.
\end{proposition}
\begin{proof}
Consider any formula of the form $A_1 , \ldots , A_n [A_{n+1} , \ldots , A_m] \groi A_{m+1}$ where the displayed occurrence of $\groi$ does not hold any occurrence of $\grot$ and suppose that there exist a $\groi$ introduction rule application with $A_1 , \ldots , A_n [A_{n+1} , \ldots , A_m] \groi A_{m+1}$ as conclusion. We first argue that, if this is the case, then \[\infer=[r]{A_{m+1}}{A_1 & \ldots & A_n [A_{n+1} & \ldots & A_m]  }\]
must be a legitimate grounding rule application. Indeed, if $A_1 , \ldots , A_n [A_{n+1} , \ldots , A_m] \groi A_{m+1}$ can be the conclusion of a $\groi$ introduction rule application, then $r$ must have been applied immediately above this rule application. 

Hence, a derivation from $A_1 , \ldots , A_n [A_{n+1} , \ldots , A_m] \groi A_{m+1}$ to $A_1 , \ldots , A_n [A_{n+1} , \ldots , A_m] \groi A_{m+1}$ which contains  only applications of  $\groi$ rules and applications of rules  explicitly mentioned in the applicability conditions of the $\groi $ introduction rule, and which contains at least one application of the elimination rule for  $\groi$ is
{\footnotesize\[\infer{A_1 , \ldots , A_n [A_{n+1} , \ldots , A_m] \groi A_{m+1}}{\infer={A_{m+1}}{\infer{A_1}{A_1 , \ldots , A_n [A_{n+1} , \ldots , A_m] \groi A_{m+1}} & \ldots & 
\infer{[A_m]}{A_1 , \ldots , A_n [A_{n+1} , \ldots , A_m] \groi A_{m+1}}}} \]}
\end{proof}

\begin{proposition}\label{prop:wdoi-tree}
The grounding tree operator enjoys weak deducibility of identicals.
\end{proposition}
\begin{proof}
Consider any formula of the form $\Gamma [  \Xi]    \groi   B$ such that the displayed occurrence of $\groi$ holds at least one occurrence of $\grot$ which is the outermost operator of a subformula of the form $(\Delta   [ \Theta]  )\grot A $ which either occurs in $\Gamma$ or in $\Xi$. 

If  $(\Delta   [ \Theta]   )\grot A $ occurs in $\Gamma$, then $\Gamma [  \Xi]    \groi   B = \Gamma ' , (\Delta   [ \Theta]   )\grot A , \Gamma ''  [  \Xi]    \groi   B $ and the derivation of $\Gamma [  \Xi]    \groi   B$ from  $\Gamma [  \Xi]    \groi   B$ which only uses rules for $\grot$ is the following:\[\infer{\Gamma ' , (\Delta   [ \Theta]   )\grot A , \Gamma ''  [  \Xi]    \groi   B}{\infer{\Delta   [ \Theta]   \groi A}{\Gamma ' , (\Delta   [ \Theta] )\grot A, \Gamma ''  [  \Xi]    \groi   B}&\infer{\Gamma ' , A, \Gamma ''   [  \Xi]    \groi   B}{\Gamma ' , (\Delta   [ \Theta]   )\grot A , \Gamma ''  [  \Xi]    \groi   B}}\]

If  $(\Delta   [ \Theta]   )\grot A $ occurs in $\Xi$, then $\Gamma [  \Xi]    \groi   B = \Gamma [  \Xi ',  (\Delta   [ \Theta]   )\grot A , \Xi '']    \groi   B $ and the derivation of $\Gamma [  \Xi]    \groi   B$ from  $\Gamma [  \Xi]    \groi   B$ which only uses rules for $\grot$ is the following:\[\infer{\Gamma [  \Xi ',  (\Delta   [ \Theta]   )\grot A , \Xi '']    \groi   B}{\infer{\Delta   [ \Theta]   \groi A}{\Gamma [  \Xi ',  (\Delta   [ \Theta]   )\grot A , \Xi '']    \groi   B}&\infer{\Gamma   [  \Xi ' , A, \Xi '']    \groi   B}{\Gamma [  \Xi ',  (\Delta   [ \Theta]   )\grot A , \Xi '']    \groi   B}}\]
\end{proof}
Three remarks are in order with respect to the proof of Proposition \ref{prop:wdoi-tree}. 
First, we must not be fooled by the fact that the derivation used in the proof only contains rules for $\grot$: this proof does not show that the strict version of  deducibility of identicals holds for the grounding tree operator and hence that $\grot$ is a logical operator. Indeed, in the base case, a grounding tree is an immediate grounding claim. Hence, strictly speaking, in order to show that the grounding tree operator enjoys weak deducibility of identicals, we need both the proof of Proposition \ref{prop:wdoi-tree} and that of Proposition \ref{prop:wdoi-gro}. 

The second remark concerns the actual constructibility of the derivation used to prove Proposition \ref{prop:wdoi-tree}. This derivation, indeed, can be constructed for any element of $\Gamma $ and $\Xi$ with $\grot$ as outermost operator. This is important since it tells us that no choice is required and that we can start from any outermost occurrence of $\grot $ to eliminate all the relevant occurrences of $\grot$ which are held by the considered occurrence of $\groi$. 

The third remark concerns a possible ambiguity of the expression {\it eliminating a grounding tree}. Indeed, technically, a grounding tree is not represented by one occurrence of an operator but by one occurrence of $\groi $ together with all occurrences of $\grot $ that this occurrence of $\groi$ holds. Coherently, an expression of  form $(\Delta   [ \Theta])\grot A $ can occur as a subformula of a formula---as, for instance, in $\Gamma , (\Delta   [ \Theta]   )\grot A [  \Xi]    \groi   B$---but is never itself a formula. Hence, when we have a formula of the  form  $\Gamma , (\Delta   [ \Theta]   )\grot A [  \Xi]    \groi   B$ we must not consider neither the occurrence of $\groi$ alone nor the occurrence of $\grot $ alone as one instance of application of the grounding tree operator. Morally, one instance of application of the grounding tree operator would include both of them, together with all other occurrences of  $\grot $ consecutively nested inside them. One could, hence, argue that, in order to eliminate one instance of a grounding tree, all occurrences of $\grot $ that are held by the outermost occurrence of $\groi $ must be eliminated. For instance, in order to eliminate the grounding tree $(p, q )\grot p\et q , (r, \non s   )\grot r\vel s \groi (p\et q )\et(r\vel s)    $  and to obtain all the immediate grounding claims composing it, we must construct the following three  derivations:\[\delta_1 \quad =\quad  \vcenter{\infer{p, q \groi p\et q  }{(p, q )\grot p\et q , (r, \non s   )\grot r\vel s \groi (p\et q )\et(r\vel s)}}\]\[ \delta _2 \quad =\quad  \vcenter{\infer{r, \non s   \groi r\vel s}{(p, q )\grot p\et q , (r, \non s   )\grot r\vel s  \groi (p\et q )\et(r\vel s)}}\]and\[\delta_3 \quad =\quad  \vcenter{\infer{p\et q , r\vel s\groi (p\et q )\et(r\vel s)}{\infer{p\et q , (r, \non s   )\grot r\vel s \groi (p\et q )\et(r\vel s)}{(p, q )\grot p\et q , (r, \non s   )\grot r\vel s \groi (p\et q )\et(r\vel s)}}}\]
Well,  even if we adopt this notion of elimination of a grounding  tree operator occurrence, weak deducibility of identicals holds for the grounding tree operator. Indeed, clearly, a decomposition similar to the one  shown in the previous example can be conducted for any  occurrence of the grounding tree operator, and the derived immediate grounding claims can be used to entirely reintroduce the original  grounding tree. For instance, for the formula considered in the previous example, the result is the following:
\[\infer{(p, q )\grot p\et q, (r, \non s   )\grot r\vel s \groi (p\et q )\et(r\vel s)}{ \deduce {r, \non s   \groi r\vel s}{\delta _2}  && \infer{(p, q )\grot p\et q ,  r\vel s \groi (p\et q )\et(r\vel s)}{ \deduce {p, q \groi p\et q  }{\delta _1 }  && \deduce{p\et q , r\vel s\groi (p\et q )\et(r\vel s)}{\delta _3}}}\]


We formally prove that it is always possible to completely decompose any non-trivial grounding tree by $\grot$ elimination rules---for trivial grounding trees, that is, immediate grounding claims, the proof of Proposition \ref{prop:wdoi-gro} is already enough---and recompose it by $\grot$ introduction rules. In order to do so, let us first define the formal notion of {\it size} of a grounding tree. Intuitively, the size of a grounding tree $\Gamma [\Delta ]\groi A$ is the number of occurrences of $\grot$ held by the outermost occurrence of $\groi$ in  $\Gamma [\Delta ]\groi A$ plus $1$.

\begin{definition}[Size $\mid\;\;\mid $ of a grounding tree]
The size $ \mid \Gamma [\Delta ]\groi A \mid $  of an immediate grounding claim $ \Gamma [\Delta ]\groi A $ is $1$. The size $\mid \Gamma [\Delta ]\groi A \mid $ of a grounding tree of the form $\Gamma [\Delta ]\groi A$ where all elements of $\Gamma $ and $\Delta $ with $\grot $ as outermost operator are $(\Gamma _1 [\Delta _1] )\grot A_1, \ldots , (\Gamma _n [\Delta _n] )\grot A_n  $ is $\mid (\Gamma _1 [\Delta _1] \groi A_1)\mid + \ldots +\mid  (\Gamma _n [\Delta _n] \groi A_n) \mid  +1 $.\end{definition}

\begin{proposition}
For any grounding tree $ \Gamma [\Delta ]\groi A$ of size greater than $1$, it is possible to derive from it all immediate grounding claims that compose it by $\grot $ elimination rules, and derive it from these immediate grounding claims by $\grot $ introduction rules.\end{proposition}
\begin{proof}The proof is by induction on the size of $ \Gamma [\Delta ]\groi A$. If $\mid  \Gamma [\Delta ]\groi A\mid = 1 $, then the statement  trivially holds. Suppose now that the statement holds for all grounding trees of size smaller than $n$, we show that it holds also for all grounding trees of size $n$. Consider any grounding tree $  \Gamma [\Delta ]\groi A$ of size $n>1$. Since $\mid   \Gamma [\Delta ]\groi A \mid >1$ there must be at least one subformula $(\Sigma [\Theta ])\grot B$ of $  \Gamma [\Delta ]\groi A$ such that either $   \Gamma [\Delta ]\groi A =  \Gamma ' , (\Sigma [\Theta ])\grot B , \Gamma '' [\Delta ]\groi A = $ or $   \Gamma [\Delta ]\groi A =   \Gamma [\Delta ' , (\Sigma [\Theta ])\grot B , \Delta ''  ]\groi A  $. 

If  $   \Gamma [\Delta ]\groi A =  \Gamma ' , (\Sigma [\Theta ])\grot B , \Gamma '' [\Delta ]\groi A $, then, clearly, $\mid \Sigma [\Theta ]\groi B \mid < n >\mid   \Gamma ' , B , \Gamma '' [\Delta ]\groi A\mid $. Hence, by inductive hypothesis, there are derivations $\gamma $ of $ \Sigma [\Theta ]\groi B $ and $\delta $  of $   \Gamma ' , B , \Gamma '' [\Delta ]\groi A $ from the immediate grounding claims that compose them only containing introduction and elimination rules for $\grot$. We can therefore construct the following derivation:\[\infer{  \Gamma ' , (\Sigma [\Theta ])\grot B , \Gamma '' [\Delta ]\groi A}{\deduce{ \Sigma [\Theta ]\groi B}{\gamma}&&\deduce{   \Gamma ' , B , \Gamma '' [\Delta ]\groi A}{\delta}}\]which verifies the statement also for $   \Gamma [\Delta ]\groi A$.

If, on the other hand,  $   \Gamma [\Delta ]\groi A =   \Gamma [\Delta ' , (\Sigma [\Theta ])\grot B , \Delta ''  ]\groi A  $, then $\mid   \Sigma [\Theta ]\groi B   \mid < n >\mid    \Gamma [\Delta ' ,  B , \Delta ''  ]\groi A \mid $. Hence, by inductive hypothesis, there are derivations $\gamma $ of $ \Sigma [\Theta ]\groi B $ and $\delta $  of $   \Gamma[\Delta ', B, \Delta '' ]\groi A $ from the immediate grounding claims that compose them only containing introduction and elimination rules for $\grot$. We can therefore construct the following derivation:\[\infer{\Gamma [\Delta ' , (\Sigma [\Theta ])\grot B , \Delta ''  ]\groi A}{\deduce{ \Sigma [\Theta ]\groi B}{\gamma}&&\deduce{   \Gamma[\Delta ', B, \Delta '' ]\groi A}{\delta}}\]which verifies the statement of the present proposition also for $   \Gamma [\Delta ]\groi A$.\end{proof}

Let us now consider mediate grounding. First of all, it is obvious that fixed a mediate grounding claim $\Gamma [\Delta ] \grom A$, there is no general strategy to construct a non-trivial derivation of  $\Gamma [\Delta ] \grom A$ from $\Gamma [\Delta ] \grom A$ by using $\grom$ rules, $\groi $ rules and, possibly, grounding rules. Indeed, for instance, there is no general way to know what grounding claims can be used to introduce the outermost occurrence of $\grom$ in $\Gamma [\Delta ] \grom A$.
And even if we suppose that the underlying relation of immediate grounding is decidable and that our grounding claim $\Gamma [\Delta ] \grom A$ is derivable, there might not be any mechanical method to find a derivation for it. For instance, if we consider any finitely axiomatisable but non-decidable formal theory,
then we would have that the immediate grounding operator meets the deducibility of identicals requirement because one-step derivability from certain premisses to a certain conclusion is decidable;  but the mediate grounding operator, on the other hand, would not meet the deducibility of identicals requirement because the derivability relation is not a decidable one. This is clearly and essentially tied to the loss of information that transitivity implies.
The mediate grounding operator, indeed, internalises a relation between a consequence and one of its mediate grounds which can be explicitly spelled out in terms of immediate grounding or can be implicitly associated to the notion of derivability in a calculus characterising the immediate grounding relation---possibly, through the notion of bar of a derivation. 
While an immediate grounding operator 
that can be characterised by a finite calculus 
expresses the existence of a rule application, a mediate grounding operator expresses the existence of a complex derivation with a certain structure. In order to account for the derivability of an immediate grounding claim of this kind, hence, it is enough to check whether there is a rule, from a finite collection of rules, that can be applied to the formulae occurring inside the grounding claim. In order to account for the derivability of a mediate grounding claim, on the other hand, a specific complex derivation of unknown size must be reconstructed, and no information concerning the structure of this derivation is provided by the mediate grounding claim. The case of grounding trees is similar to that of mediate grounding, but with an essential difference: a grounding tree expresses the existence of a complex derivation with a certain structure---which is specified by the claim---and containing certain formulae---which, again, are specified by the claim. Since a grounding tree explicitly provides all the information required to reconstruct the complex derivation that justifies the derivability of the claim itself, it is easy to reconstruct such a derivation and to reduce the derivability of grounding trees to that of immediate grounding claims. This difference between mediate grounding, on the one side, and immediate grounding and grounding trees, on the other side, is clearly related to the fact that obtaining a  mediate grounding claim on the basis of a set of immediate grounding claims by transitivity implies a considerable loss of information with respect to the original set of immediate grounding claims. The problem concerning  the loss of information implied by taking the transitive closure of immediate grounding in order to define mediate grounding is also of philosophical interest, as the discussion on the matter that can be found in \cite{sch12, lit13, rav13} witnesses.

\section{Conclusions}
\label{sec:conclusions}

We have introduced three sets of inferential rules that can be used to define the behaviour of grounding operators of three different kinds on the basis of a generic grounding calculus: an operator for immediate grounding, an operator for mediate grounding---corresponding  to the transitive closure of the immediate grounding one---and a grounding tree operator---that is, an operator  that enables us to internalise chains of immediate grounding claims without loosing any information about them. We have characterised the behaviour of these operators and studied their proof-theoretical properties. 

In particular, we have shown that all three operators enjoy local detour eliminability since detour reductions for all of them can be defined. Nevertheless, we have also shown that while the schematic behaviour of the rules for the immediate grounding operator $\groi$ and for the grounding tree operator $\grot$ enable us to generalise existing normalisation results for grounding calculi---as the generalisation of the normalisation result in \cite{gen21} shows---and hence to show global detour eliminability with respect to grounding calculi as well; the rules for the mediate  grounding operator  $\grom$ pose serious technical  problems with respect to global detour eliminability results which are, as we argued, related to the conceptual features of mediate grounding which the $\grom$ operator is meant to  formalise.

We have also considered the deducibility of identicals criterion, which, along with the detour eliminability criterion, has been proposed as a test for logicality. We have shown that all three operators fail this test and therefore argued that there is strong technical evidence against the claim that grounding operators are logical operators. The philosophical reasons of this failure have been discussed along with a connection between the hyperintensionality of grounding and the non-logicality of grounding operators.

In an attempt to distinguish between the logicality of the considered  operators and the balanced interplay between their introduction and elimination rules, we have then defined a weaker version of the deducibility of identicals criterion that takes into consideration the hyperintensional nature of grounding. 
By the weaker criterion, we have shown that, while the rules for the immediate grounding and grounding tree operator display the balance between introductions and eliminations required to meet the  weak deducibility of identicals criterion, the rules for the mediate grounding operator do not. We discussed the ill behaviour of the mediate grounding operator both with respect to global detour eliminability and with respect to weak deducibility of identicals in light of the fact that the definition of mediate grounding by taking the transitive closure of immediate grounding implies a possibly considerable loss of information. A possible parallel with the philosophical problems posed by the transitivity of grounding has been proposed.


The presented work raises two general questions. The first concerns the suitability of mediate grounding as a notion of grounding. While the presented results are not meant to constitute conclusive evidence of specific features of particular grounding relations, but only to enlighten the characteristics shared by a very general class of formal grounding operators; the technical shortcomings of the mediate grounding operator studied here seem to point at very specific philosophical issues that also concern informal notions of mediate grounding. The technical results presented here also indicate very clearly, though, that these shortcomings are essentially tied to specific features of the underlying immediate grounding relation, and hence that do not necessarily bear relevance to all notions of grounding. Hence, a more specific investigation of the relations between transitivity and decidability of grounding relations would be of great interest. The second question concerns the possibility of an argument of general validity establishing the exact  connections between hyperintensionality and logicality criteria. While such an argument is at the moment impossible, since it requires a general proof-theoretical characterisation of hyperintensionality; the philosophical attention that hyperintensional notions are receiving lately certainly makes the development of suitable formal methods an endeavour of great interest.

%

\bmhead{Acknowledgments}
\section*{Declarations}

\bibliography{bib-grounding}


\end{document}